\documentclass[11pt]{amsart}
\usepackage[top=1in, bottom=1.25in, left=1.25in, right=1.25in]{geometry}
\usepackage{amssymb,amsmath,amsthm,mathrsfs}
\usepackage{enumitem}

\usepackage{mmacells}
\usepackage{url}
\usepackage{cite}
\usepackage{amsfonts}
\usepackage{verbatim}
\usepackage{graphicx}
\usepackage{caption} 
\usepackage[all]{xy}
\usepackage{colonequals}

\theoremstyle{plain}
\newtheorem{thm}{Theorem}[section]
\newtheorem*{thm*}{Theorem}

\newtheorem{prop}[thm]{Proposition}

\newtheorem{remark}[thm]{Remark}

\renewcommand{\Re}{\mathrm{Re}}
\renewcommand{\Im}{\mathrm{Im}}

\def\C{{\mathbb{C}}}
\def\Z{{\mathbb{Z}}}
\def\T{{\mathbb{T}}}
\def\R{{\mathbb{R}}}

\def\SL{{\text{\rm{SL}}}}
\def\Sym{{\text{\rm{Sym}}}}
\def\Sp{{\text{\rm{Sp}}}}
\def\Aff{{\text{\rm{Aff}}}}
\def\Aut{{\text{\rm{Aut}}}}
\def\Zar{{\text{\rm{Zariski}}}}
\usepackage{color}
\definecolor{purple}{rgb}{0.5,0,1}
\definecolor{darkgreen}{rgb}{0.1,0.4,0.2}
\definecolor{darkyellow}{rgb}{0.6,0.6,0.2}

\newenvironment{rodolfo}{
  \medskip
\begin{color}{darkgreen}
    \textcolor{darkyellow}{\textbf{Rodolfo:}} 
}{
\end{color}
  \medskip
}

\newenvironment{anthony}{
  \medskip
\begin{color}{red}
    \textcolor{red}{\textbf{Anthony:}} 
}{
\end{color}
  \medskip
}

\title[Kontsevich--Zorich monodromy groups of translation covers]{Kontsevich--Zorich monodromy groups of translation covers of some platonic solids}
\author{Rodolfo Gutiérrez-Romo}
\address{Rodolfo Gutiérrez-Romo: Centro de Modelamiento Matemático, CNRS-IRL 2807, Universidad de
Chile, Beauchef 851, Santiago, Chile.}
\email{g-r@rodol.fo}
\author{Dami Lee}
\address{Dami Lee:  Department of Mathematics, University of Washington, Seattle, WA 98115, USA}
\email{damilee@uw.edu}
\author{Anthony Sanchez}
\address{Anthony Sanchez: Department of Mathematics, University of California San Diego, 9500 Gilman Dr, La Jolla,
CA 92093, USA}
\email{ans032@ucsd.edu}
\date{}


\begin{document}
\maketitle 

\begin{abstract}
We compute the Zariski closure of the Kontsevich--Zorich monodromy groups arising from certain square-tiled surfaces that are geometrically motivated. Specifically we consider three surfaces that emerge as translation covers of platonic solids and quotients of infinite polyhedra, and show that the Zariski closure of the monodromy group arising from each surface is equal to a power of $\SL(2, \R)$.

We prove our results by finding generators for the monodromy groups, using a theorem of Matheus--Yoccoz--Zmiaikou \cite{MYZ} that provides constraints on the Zariski closure of the groups (to obtain an ``upper bound''), and analyzing the dimension of the Lie algebra of the Zariski closure of the group (to obtain a ``lower bound'').

Moreover, combining our analysis with the Eskin--Kontsevich--Zorich formula \cite{MR3270590}, we also compute the Lyapunov spectrum of the Kontsevich--Zorich cocycle for said square-tiled surfaces.

\end{abstract}


\section{Introduction and main results}
A \emph{translation surface} is a Riemann surface $X$ endowed with a nonzero holomorphic 1-form $\omega$. These surfaces form a moduli space with a natural $\SL(2, \R)$-action by postcomposition with coordinate charts, and the central subaction by the diagonal subgroup of $\SL(2, \R)$ is known as the \emph{Teichmüller flow}. An \emph{orbit closure} is a closed $\SL(2,\R)$-invariant subset of the moduli space.

Given an orbit closure $\mathcal{M}$ with underlying topological surface $S$, the Hodge bundle is defined as (orbifold) vector bundle induced by $H^1(S; \R)$. The \emph{Kontsevich--Zorich cocycle} is the dynamical cocycle over the Hodge bundle induced by the action of a suitable group. Different versions of this cocycle can be found in the literature, depending on whether all orbits, $\SL(2, \R)$-orbits or only Teichmüller flow orbits are considered. Moreover, Avila--Eskin--M\"{o}ller \cite{MR3717086} prove that the cocycle admits a semisimple decomposition, meaning that it can be split into $\SL(2, \R)$-invariant complementary subbundles.

The \emph{Kontsevich--Zorich monodromy groups} encode the homological data of translation surfaces along $\SL(2,\R)$-orbits, that is, how the Kontsevich--Zorich cocycle acts on the Hodge bundle, and have been studied extensively. General questions concern the algebraic nature of Kontsevich--Zorich monodromy groups. For example, Filip \cite{MR3619303} provided constraints for the Zarsiki closure of the Kontsevich--Zorich monodromy groups of translation surfaces corresponding to strongly irreducible subbundles and showed that, at the level of real Lie algebra representations, they must belong to a finite list of matrix families. On the other hand, Matheus--Yoccoz--Zmiaikou \cite{MYZ} proved constraints for these groups in the case of square-tiled surfaces, that is, covers of a torus branched over a single point. Realizability of the matrix families has been studied by several authors \cite{MR3959361, MR3959355, MR3743240}. Other algebraic questions concern whether the Kontsevich--Zorich monodromy groups are arithmetic (see Hubert--Matheus \cite{MR4120783} for the existence of an arithmetic monodromy group) and how frequently this is the case (see Bonnafoux et.\ al \cite{long}).

Lyapunov exponents of the Kontsevich--Zorich cocycle measure the growth rate of homological data along the Teichmüller flow. Since the Kontsevich--Zorich cocycle is symplectic in each piece $E$ of the semisimple decomposition, the Lyapunov spectrum corresponding to $E$ is of the form $\lambda_1 \geq \lambda_{2} \geq \dotsb \geq \lambda_d \geq 0 \geq -\lambda_d \geq \dotsb \geq -\lambda_2 \geq -\lambda_1,$
where $2d$ is the (real) dimension of $E$. An important particular case is when $E$ is the subbundle whose fiber at $(X, \omega)$ is $\langle \Re(\omega), \Im(\omega)\rangle$, which is known as the \emph{tautological plane} and always carries a Lyapunov spectrum of $\{\pm 1\}$.

There exists a relation between the Lyapunov spectrum of an orbit closure and the algebraic nature of its Kontsevich--Zorich monodromy group which is not yet fully understood. For example, Filip's classification can be refined in the presence of zero Lyapunov exponents, and, if the Lyapunov spectrum corresponding to a strongly irreducible subbundle is simple (that is, if all Lyapunov exponents are distinct), then the Zariski closure of the group is known to be $\Sp(2d, \R)$, where $2d$ is the (real) dimension of the subbundle.

For square-tiled surfaces, an algebraic criterion for simplicity was found in Matheus--Möller--Yoccoz \cite{MR3402801}, and a coding-free criterion was provided in Eskin--Matheus \cite{MR3424657}. Furthermore, it was shown in Eskin--Kontsevich--Zorich \cite{MR2820564} that there exist square-tiled surfaces where coincidences of Lyapunov exponents occur across distinct, symplectically-orthogonal, irreducible subbundles. This fact is interesting because there are no known reasons why simplicity should occur or not when considering the joint Lyapunov spectrum corresponding to the direct sum of such subbundles. Indeed, simplicity criteria (including the previous ones for square-tiled surfaces and also the general criterion due to Avila--Viana \cite{MR2350698, MR2316268}) concern a single strongly irreducible subbundle, and to our knowledge there are no general results relating exponents in different pieces of the semisimple decomposition of the Kontsevich--Zorich cocycle.

While individual Lyapunov exponents are in general extremely hard to compute, Eskin--Kontsevich--Zorich \cite{MR3270590} found a formula for the \emph{sum} of all positive Lyapunov exponents of an orbit closure. In particular, the resulting sum is always rational and, for the case of square-tiled surfaces, this rational number can be computationally found.

In this article, we consider the Kontsevich--Zorich monodromy groups of square-tiled surfaces that arise as the translation cover of platonic solids and quotients of infinite polyhedra.
Translation covers of platonic solids and polyhedral surfaces are useful because the extra structure of the cover has been used to shed light on the underlying solid or surface (see for example Athreya--Aulicino \cite{AA} where they used the translation cover to find a closed path
between the vertices on the dodecahedron).

Our main result identifies the Zariski closure of the Kontsevich--Zorich monodromy group of the translation cover of some platonic and polyhedral surfaces. 

\begin{thm*} The Zariski closure of the Kontsevich--Zorich monodromy group (restricted  to the zero-holonomy subspace) of the translation cover of
\begin{enumerate}
    \item  the octahedron is $\SL(2,\R)$;
    \item the cube is $ \SL(2,\R)^3$; and
    \item the mutetrahedron is $ \SL(2,\R)^4$.
\end{enumerate}
\end{thm*}

\begin{remark}\label{rem: common covers}
Theorem 1.1 in \cite{AL} tells us the translation covers of the octahedron (cube, respectively) and the translation cover of the compact quotient of Octa-8 (Mucube, respectively) are isometric with their intrinsic metrics. (See Section~\ref{sec: unfolding} for the definition of these surfaces.) Hence our results naturally apply to translation covers of the quotient of Octa-8 and Mucube.
\end{remark}

\begin{remark}
In all cases but the translation cover of the octahedron, the zero-holonomy subspace of the first homology group can be further decomposed into irreducible pieces, and our methods allow us to compute the monodromy group restricted to any such component. Nevertheless, in all cases but the translation cover of the mutetrahedron, such decompositions do not yield strongly irreducible subbundles.
\end{remark}

It is worth highlighting that in all of our examples, the Kontsevich--Zorich monodromy groups are never Zariski dense in the ambient symplectic group. It would be interesting to see if this is true for all platonic solids and to understand how this is related to the symmetries of the underlying platonic solid.

In addition, we compute the Lyapunov spectrum of the translation covers we consider and show they are not simple. Furthermore, we even observe coincidences in exponents corresponding to distinct, symplectically-orthogonal subbundles for the cover of the cube and mutetrahedron. As was previously mentioned, there are no general results relating Lyapunov exponents across such distinct subbundles, so our work adds up to the literature of known examples where such coincidences between Lyapunov exponents exist. To the best of our knowledge, this phenomenon was first observed by Eskin--Kontsevich--Zorich as they studied cyclic covers of square-tiled surfaces \cite{MR2820564}.

More precisely, we show the following:

\begin{prop}
Counting multiplicities, the positive Lyapunov spectrum of the translation cover of 
\begin{enumerate}
    \item the octahedron is $\{1\} \cup \{1/2,1/2,1/2\}$;
    \item the cube is $\{1\} \cup \{2/3,2/3,2/3\} \cup \{1/3,1/3,1/3\} \cup \{1/3,1/3\}$; and
    \item the mutetrahedron is $\{1\} \cup \{1/2\} \cup \{1/2\} \cup \{1/2\} \cup \{1/2\}$;
\end{enumerate}
where the unions indicate the exponents corresponding to distinct, symplectically-orthogonal, irreducible pieces of the Hodge bundle.
\end{prop}

In Section~\ref{sec: intro}, we summarize the key objects and key tools that appear in this paper. In Sections \ref{sec: octa}, \ref{sec: cube}, and \ref{sec: mut}, we compute the Kontsevich--Zorich monodromy group of the translation cover of the octahedron, cube, and mutetrahedron, respectively. 

\subsection{Strategy} Our strategy to compute the Zariski closures is as follows. First, we analyze the representation-theoretic properties of the automorphism group of each surface and use the constraints found by Matheus--Yoccoz--Zmiaikou \cite{MYZ} to obtain an ``upper bound'' on such Zariski closures. This amounts to studying the irreducible representations of the automorphism group that arise by the homological action of the automorphism group and determining its nature, that is, whether it is real, complex or quaternionic, by examining the centralizer of the representation inside all endomorphisms of the vector space on which it acts. Then, we compute the so-called \emph{isotypical components} of each irreducible representation, that is, the direct sum of subspaces where the automorphism group acts as said irreducible representation. When this is done, we can apply the results by Matheus--Yoccoz--Zmiaikou \cite{MYZ} directly.

To obtain a ``lower bound'', we use generators for the Veech group of each surface to obtain enough linearly independent elements inside the Lie algebra of the Zariski closure of each monodromy group. This produces a lower bound on the dimension of said Zariski closure, and we see that this lower bound is equal to the dimension of our upper bound, concluding the proof.

To compute the Lyapunov spectrum of the surfaces, we make use of the Eskin--Kontsevich--Zorich formula multiple times to find the sum of the positive Lyapunov exponents corresponding to different subbundles coming from quotients of the original surface by subgroups of the automorphism group. These subbundles are sometimes not $\SL(2, \R)$-invariant, but the resulting Lyapunov exponents get carried over to other pieces of the $\SL(2, \R)$-invariant subbundle containing them. In this way, we are able to obtain enough linear equations to solve for each individual Lyapunov exponent.

\subsection{Acknowledgements} The authors would like to thank Jayadev Athreya for proposing this project. The authors thank Vaibhav Gadre for helpful conversations concerning Lyapunov exponents. The authors would also like to thank Carlos Matheus Santos and Anton Zorich for helpful insights throughout the project, and A.S.\ would especially like to thank C.M.S.\ for teaching him about this field of study and for being a sponsoring scientist as a part of the NSF Graduate Research Opportunities Worldwide program. During this project, Gutiérrez-Romo was supported by Centro de Modelamiento Matemático (CMM), ACE210010 and FB210005, BASAL funds for centers of excellence from ANID-Chile, and the FONDECYT Iniciación program under Grant No.\ 11190034; Lee was partially supported by the NSF under Grant No.\ DMS-1440140; and Sanchez was supported by the NSF Postdoctoral Fellowship under Grant No.\ DMS-2103136.
 
\section{Preliminaries}\label{sec: intro}
\subsection{Translation surface and moduli space}

A \emph{translation surface} is a polygon in $\C$ with sides identified, in pairs, by translation in such a way that the resulting topological surface is orientable. The cone angle at any given point is always an integer multiple of $2\pi.$ This naturally gives rise to a holomorphic 1-form; an order-$k_i$ zero of $\omega$ corresponds to a cone point of cone angle $2 \pi (k_i + 1).$ In other words, a translation surface is a Riemann surface with a nonzero holomorphic 1-form, which we denote $(X, \omega).$ The genus of the surface can be recovered from the order of the zeros by $\sum k_i = 2 g - 2.$

Given a genus $g,$ consider the set of pairs $(X, \omega)$ and denote this set by $\mathcal{L}_g.$ This set is equivalent to the set of abelian differentials on compact Riemann surfaces of genus $g.$ The \emph{moduli space of abelian differentials} is defined by $\mathcal{H}_g \colonequals \mathcal{L}_g/\Gamma_g$ where $\Gamma_g \colonequals \textrm{Diff}^+(S) / \textrm{Diff}_0^+(S)$ is the \emph{mapping class group} of genus $g$ surfaces $S.$

The identification between the set of abelian differentials and translation structures allows us to consider $\text{GL}^+(2,\R)$-actions on $\mathcal{L}_g.$ The $\text{GL}^+(2,\R)$-action on $\mathcal{H}_g$ preserves the zeros and their orders, 
and furthermore the $\SL(2,\R)$-action preserves the area of each translation surface. We define $\mathcal{H}_g^{(1)}$ as the moduli space of abelian differentials with unit area. 

$\SL(2,\R)$-actions on $(X,\omega)$ yield an orbit in the moduli space. If this orbit is closed, we say that $(X,\omega)$ is a Veech surface and define the \emph{Veech group} of $(X,\omega)$ as the stabilizer in $\SL(2,\R).$ We denote the Veech group of $(X,\omega)$ by $\SL(X,\omega)$ or $\SL(X).$

\subsection{Square-tiled surfaces}

A \emph{square-tiled surface} (or \emph{origami}) is a translation surface defined by a finite number of unit squares where identification of edges can be viewed by two permutations $\sigma_h, \sigma_v$ on the set of squares. Namely, the right side of square $i$ is identified with the left side of square $\sigma_h(i)$ and the top side of square $i$ is identified with the bottom side of square $\sigma_v(i).$ The surface is defined uniquely up to simultaneous conjugation. That is, $(\sigma_h,\sigma_v)$ and $(\varphi \sigma_h \varphi^{-1}, \varphi \sigma_v \varphi^{-1})$ define the same origami for any $\varphi \in S_N.$ Since a square-tiled surface can be defined completely combinatorically, its Veech group can be computed by \texttt{SageMath} and \texttt{surface\_dynamics} \cite{DFL}. One indication of how this is done is through the use of the matrices $$
    T=
  \begin{pmatrix}
     1 & 1\\
     0 & 1
  \end{pmatrix}
  \text{ and }
  S =
  \begin{pmatrix}
     1 & 0\\
     1 & 1
  \end{pmatrix}.$$
For a square-tiled surface, the Veech group is necessarily contained in $\SL(2,\Z)$. Since  $\SL(2,\Z)$ is generated by $T$ and $S$, understanding the Veech group amounts to understanding the action of $T$ and $S$. The action of $T$ and $S$ on a square-tiled surface $X = (\sigma_h,\sigma_v)$ can be defined in a purely combinatorial manner via the formulae $T(X) = (\sigma_h,\sigma_v\sigma_h ^{-1})$ and $S(X) = (\sigma_h\sigma_v ^{-1},\sigma_v)$.

Another way of defining a square-tiled surface $(X, \omega)$ is as a branched covering $p\colon (X,\omega) \to \T^2$ branched only at $0 \in \T^2$ and $\omega = p^*(dz).$ Then the \emph{automorphism group} of a square-tiled surface is the group of homeomorphisms $f$ on $X$ which satisfy $p \circ f = p.$



\subsection{Actions on homology and the Kontsevich--Zorich monodromy group}

As a square-tiled surface may have nontrivial automorphisms, the Hodge bundle over its $\SL(2,\R)$-orbit (and hence the corresponding Kontsevich--Zorich cocycle) is orbifoldic. However, by considering a finite cover of $\SL(X)$ given by the affine diffeomorphisms $\Aff(X)$, we can destroy the orbifoldic nature and obtain a genuine cocycle. While the Hodge bundle is usually defined as having $H^1(X; \R)$ as a fiber, by Poincaréduality we will equivalently consider $H_1(X; \R)$ as the fiber.

Let $\tilde \alpha\colon \Aff(X,\omega)\to \Sp(H_1(X;\R))$ denote the representation arising from the action of  $\Aff(X,\omega)$ on the absolute homology group $H_1(M,\R)$. 

The homology group has a natural splitting $$H_1(X;\R)=H_1^{\mathrm{st}}(X)\oplus H_1 ^{(0)}(X)$$
where $H_1^{\mathrm{st}}(X)$ is a 2-dimensional subspace known as the \emph{tautological plane} spanned by the real and imaginary parts of the implicit abelian differential $\omega$ and the space $H_1 ^{(0)}(X)$ is the $2g-2$-dimensional orthogonal complement (with respect to the intersection form) given by the zero-holonomy subspace. That is,
 $$H_1 ^{(0)}(X)=\left\{\gamma\in H_1(X;\R):\int_\gamma \omega=0\right\}.$$ 
 The image $\tilde \alpha(\Aff(X))$ respects this decomposition. Moreover, $\tilde \alpha|_{H_1^{\mathrm{st}}(X)}$ can be identified with the Veech group $\SL(X)$ (see e.g.\ the second paragraph of page 7 of\cite{MR3743240}) and as such is well understood. Thus, understanding $\tilde \alpha$ amounts to understanding how the zero-holonomy subspace decomposes under the action of $\Aff(X)$. We will refer to the restriction  of the image $\alpha = \tilde \alpha|_{H_1 ^{(0)}(X)}$ as the \emph{Kontsevich--Zorich monodromy group}.

\subsection{Unfolding of platonic surfaces}\label{sec: unfolding}

A polyhedral surface in $\R^3$ is a surface tiled by polygons so that the polygons (faces) are either disjoint or share an edge or a vertex. Hence, a polyhedral surface is naturally equipped with a cone metric. A point at the interior of a face or an edge is a trivial cone point as the cone angle is $2 \pi.$ Non-trivial cone points occur possibly at the vertices of the surface. If all cone angles are integer multiples of $2 \pi,$ then the surface has a translation structure. If the cone angles at cone points are $2 \pi \left(\frac{k_i}{q} + 1\right)$ for some integer $q,$ then one can consider its $q$-fold cover, the \emph{unfolding} (\emph{translation cover} or \emph{spectral curve}), branched at the vertices. Table~\ref{tab: platonic surfaces} shows the genus of the unfoldings (as square-tiled surfaces) studied in \cite{AAH} and \cite{AL}. 

\begin{table}[htbp]
\centering
\begin{tabular}[h]{c|c||c|c}
\hline 
Polyhedron & \begin{tabular}{@{}c@{}}Genus of the \\ translation cover\end{tabular} & \begin{tabular}{@{}c@{}}Quotient of the \\ infinite polyhedron\end{tabular} & \begin{tabular}{@{}c@{}}Genus of the \\ translation cover\end{tabular} \\ \hline \hline
Tetrahedron & 1 & Octa-8 $\{3,12\}$ & 4 \\ \hline
Octahedron & 4 & Mutetrahedron $\{6,6\}$ & 5 \\ \hline 
Cube & 9 & Mucube $\{4,6\}$ & 9 \\ \hline
Icosahedron & 25 & Octa-4 $\{3,8\}$ & 19 \\ \hline
Dodecahedron & 81 & Muoctahedron $\{6,4\}$ & 19 \\ \hline
& & Truncated Octa-8 $\{4,5\}$ & 49 \\ \hline 
\end{tabular}
	\caption{Genus of translation covers of polyhedral surfaces.}
	\label{tab: platonic surfaces}
\end{table}

The surfaces on the right half of Table~\ref{tab: platonic surfaces} are quotients of triply periodic polyhedral surfaces. We denote a polyhedral surface by Schl{\"a}fli symbols $\{p, q\}$ if it is tiled by regular Euclidean $p$-gons and all vertices are $q$-valent. Triply periodic polyhedral surfaces are infinite surfaces that are invariant under a rank-three lattice in $\R^3.$ Their quotients under the lattice are compact Riemann surfaces with natural polyhedral cone metric. In this paper, we refer to the compact quotients by the name of the original infinite surfaces.

\begin{figure}[htbp] 
\centering
\begin{minipage}{.33\textwidth}
	\centering
	\includegraphics[width=\textwidth]{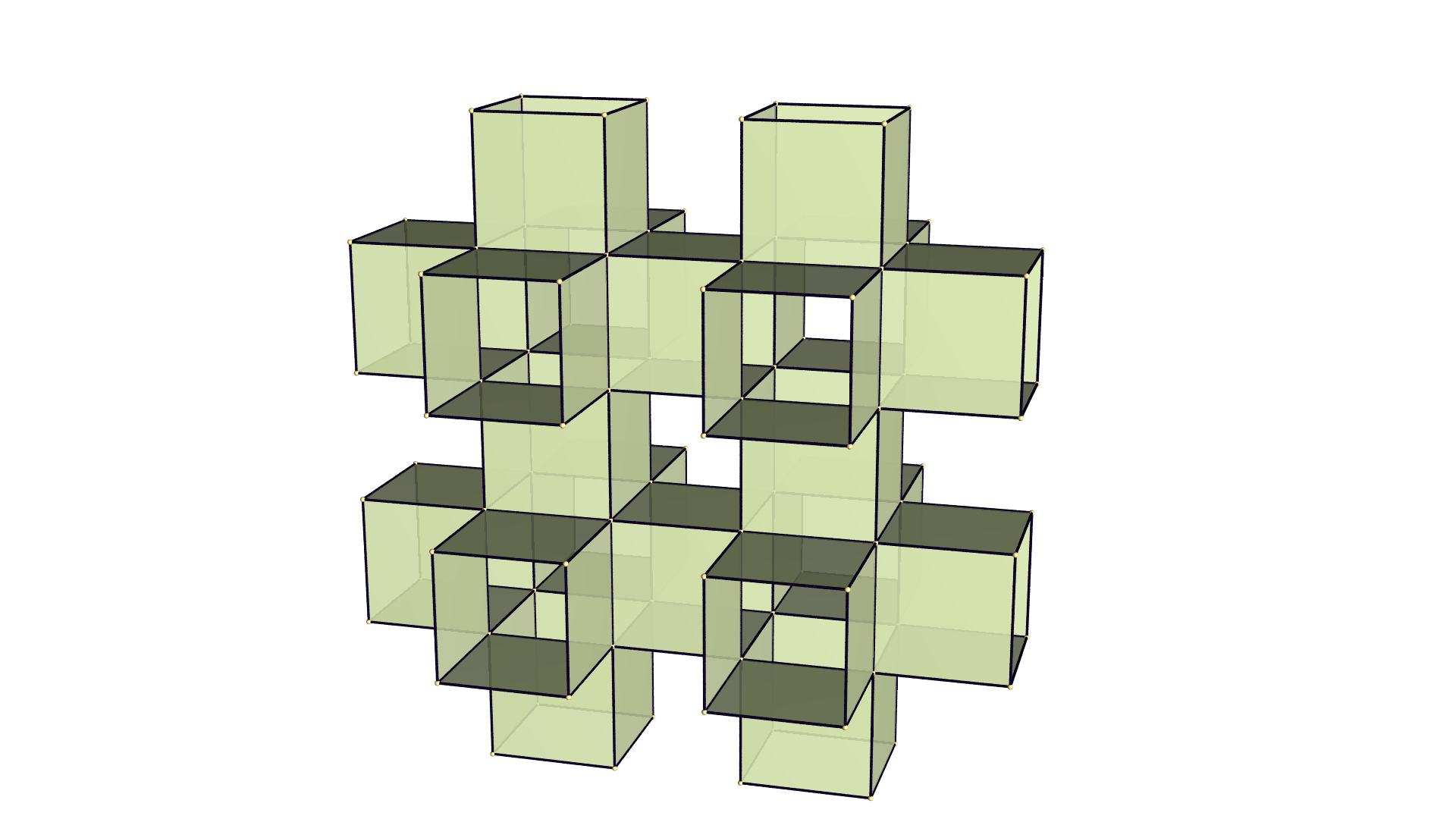}
	\end{minipage}%
\begin{minipage}{.33\textwidth}
	\centering
	\includegraphics[width=\textwidth]{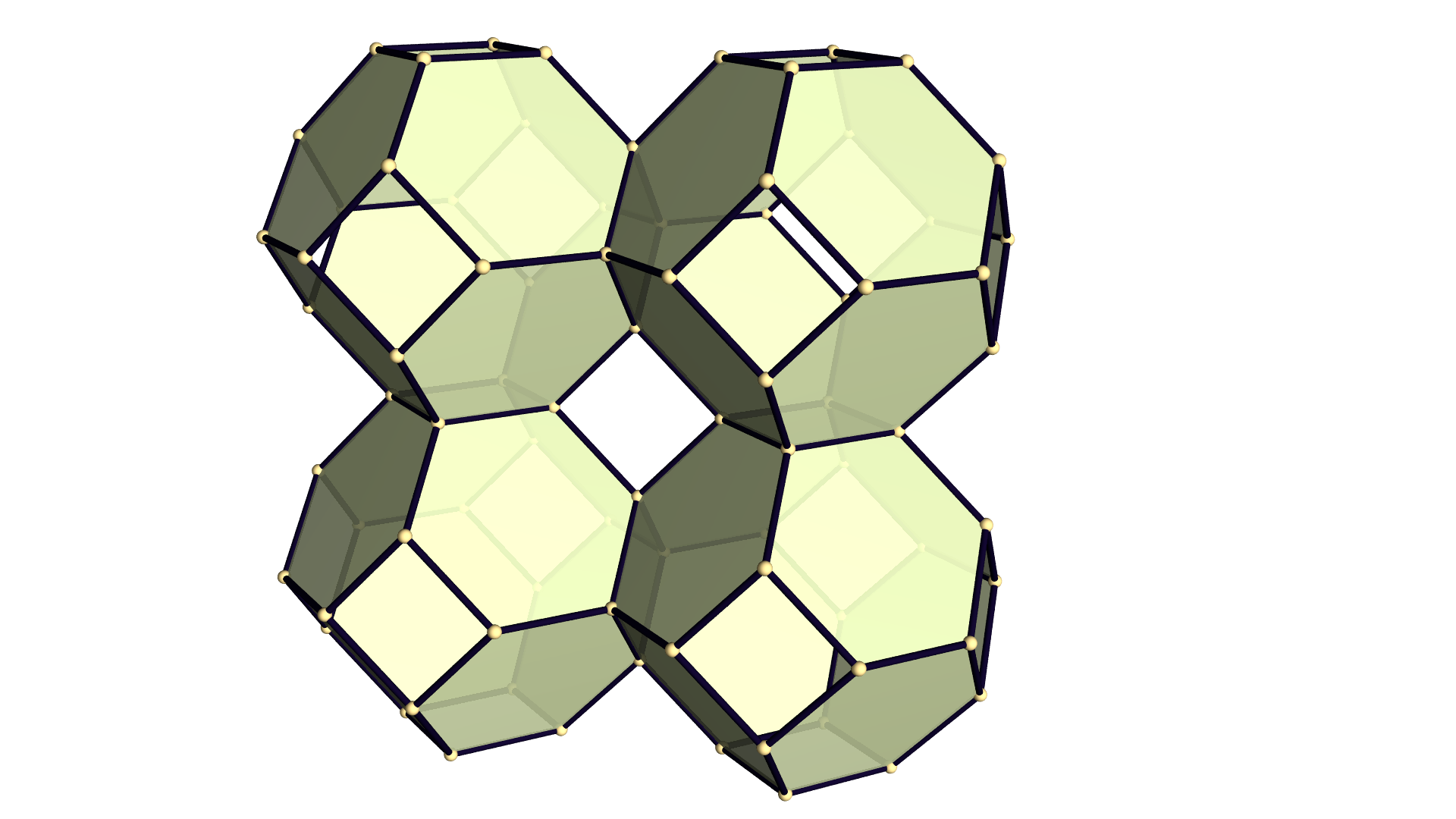}
	\end{minipage}%
\begin{minipage}{.33\textwidth}
	\centering
	\includegraphics[width=\textwidth]{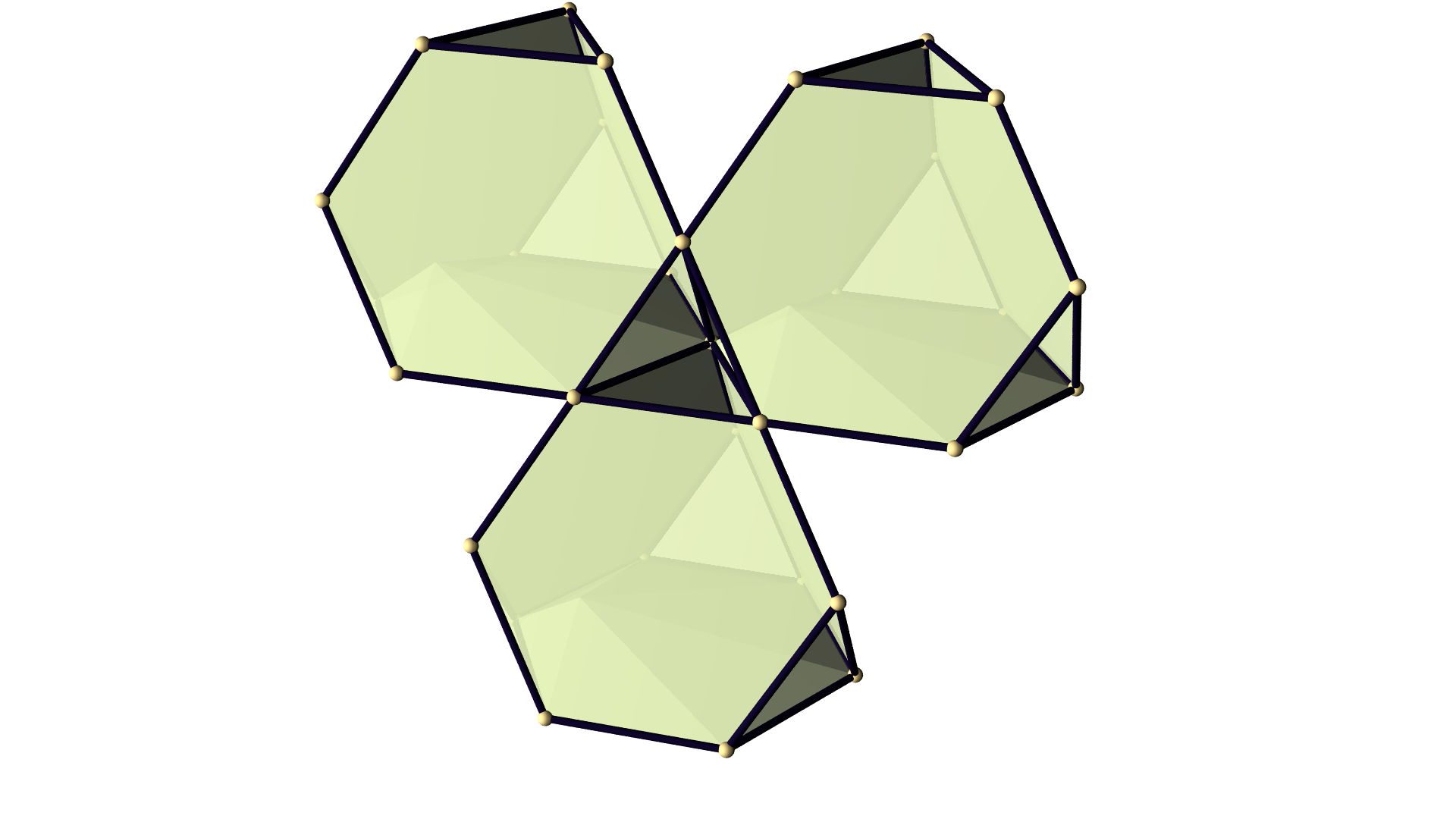}
	\end{minipage}
	\caption{A subset of the infinite polyhedra Mucube, Muoctahedron, and Mutetrahedron. Adapted from \cite{L}.}
	\label{fig: platonic surfaces 1}
\end{figure}

\begin{figure}[htbp] 
\centering
\begin{minipage}{.33\textwidth}
	\centering
	\includegraphics[width=\textwidth]{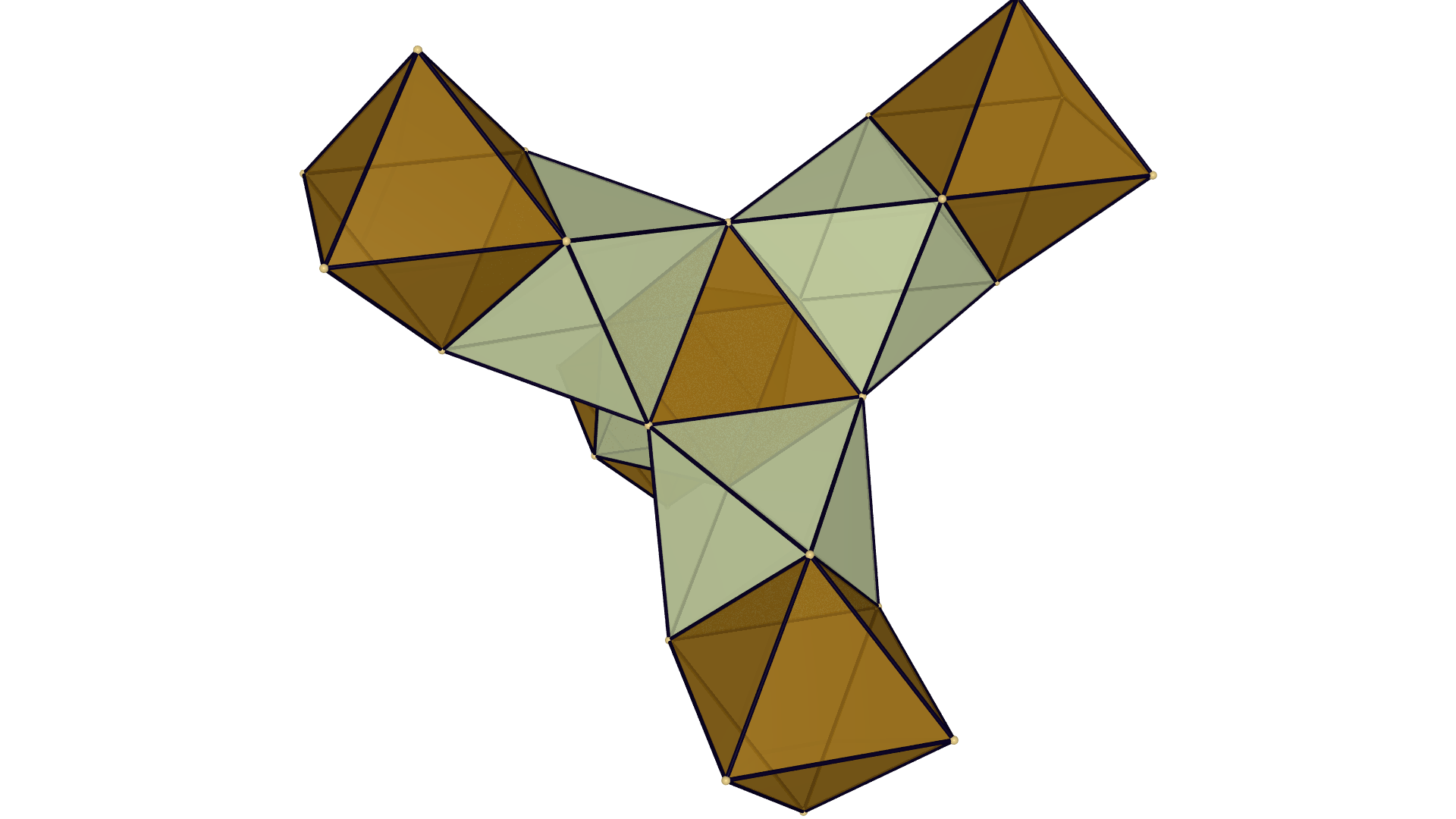}
	\end{minipage}%
\begin{minipage}{.33\textwidth}
	\centering
	\includegraphics[width=\textwidth]{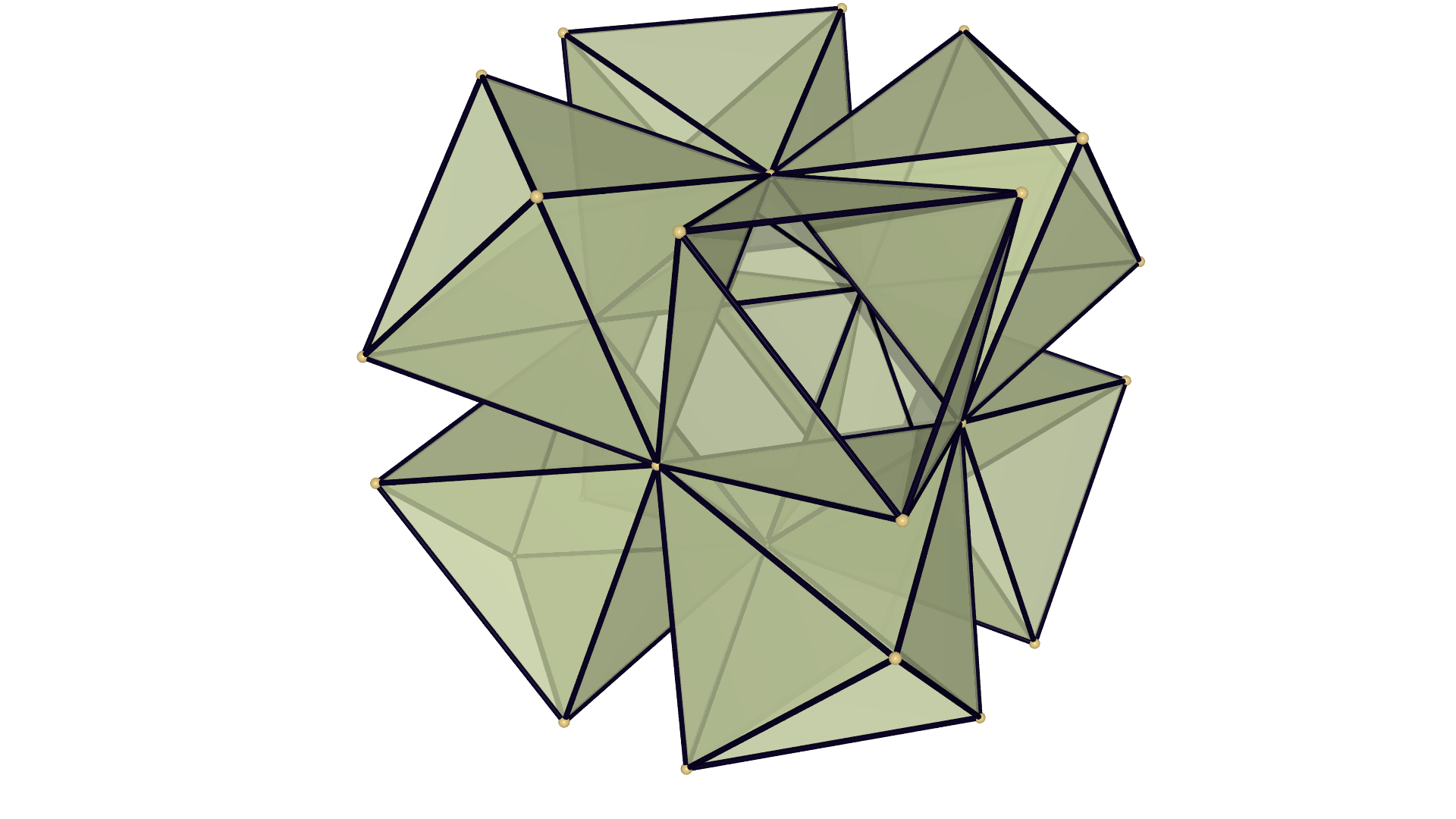}
	\end{minipage}%
\begin{minipage}{.33\textwidth}
	\centering
	\includegraphics[width=\textwidth]{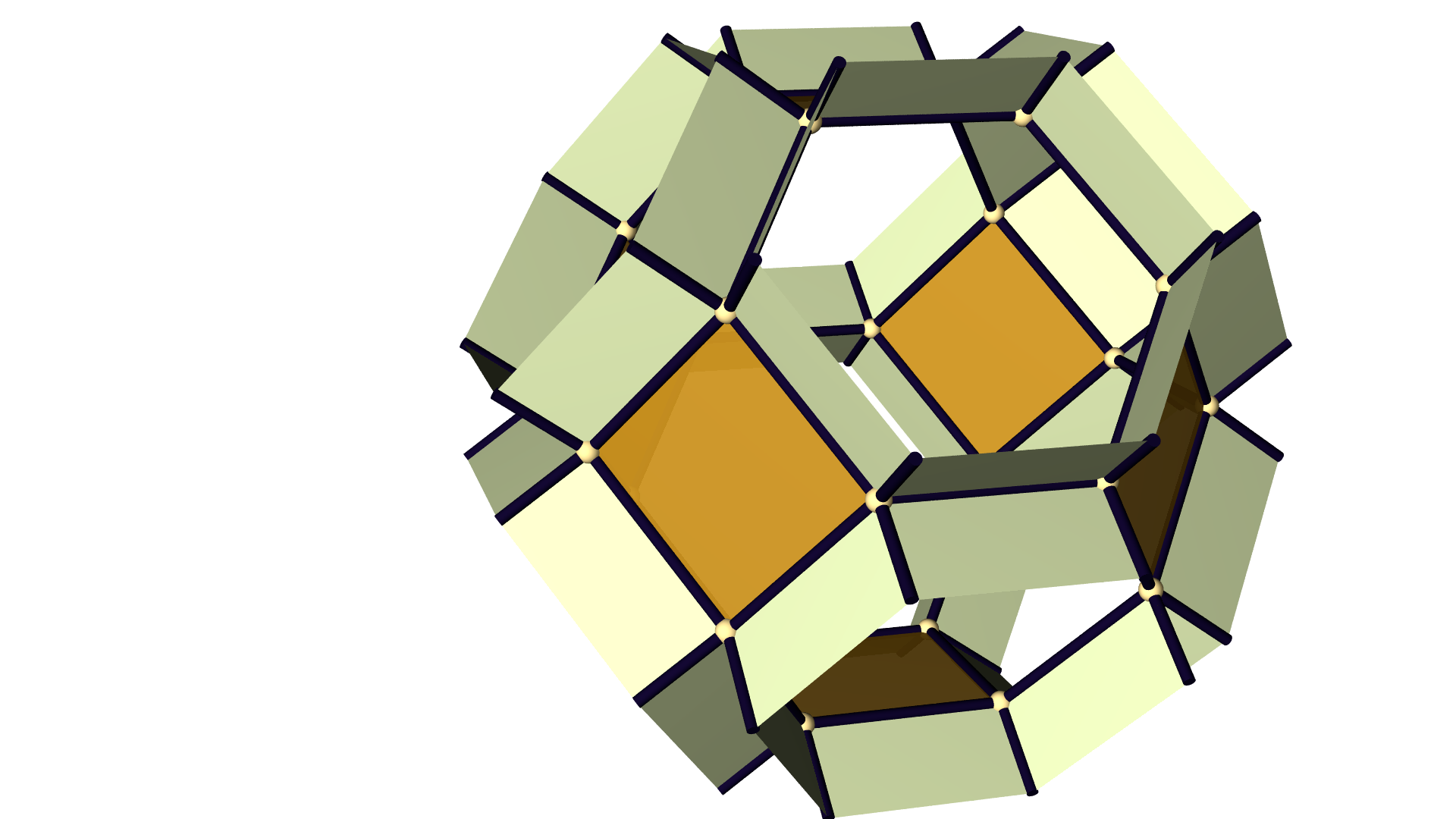}
	\end{minipage}
	\caption{A subset of the infnite polyhedra Octa-4, Octa-8, and Truncated Octa-8. Adapted from \cite{L}.}
	\label{fig: rad}
\end{figure}

As done in \cite{AAH} and \cite{AL}, the translation covers of all these surfaces (except the dodecahedron) can be studied under the same tools that are used in the study of square-tiled surfaces. To a triangle- or hexagon-tiled surface, we apply appropriate shear maps (Figure~\ref{fig: shear}) that map it to a square-tiled surface. By \cite{GJ}, these surfaces have Veech groups which are finite index subgroups of $\SL(2, \mathbb Z)$, the Veech group of the square torus. The Veech group of the triangle-tiled or hexagon-tiled surface is conjugate to the Veech group of the associated square-tiled surface. We will again abuse notation and call the translation covers by their underlying quotient surface.

\begin{figure}[htbp] 
	\centering
    \includegraphics[width=4.5in]{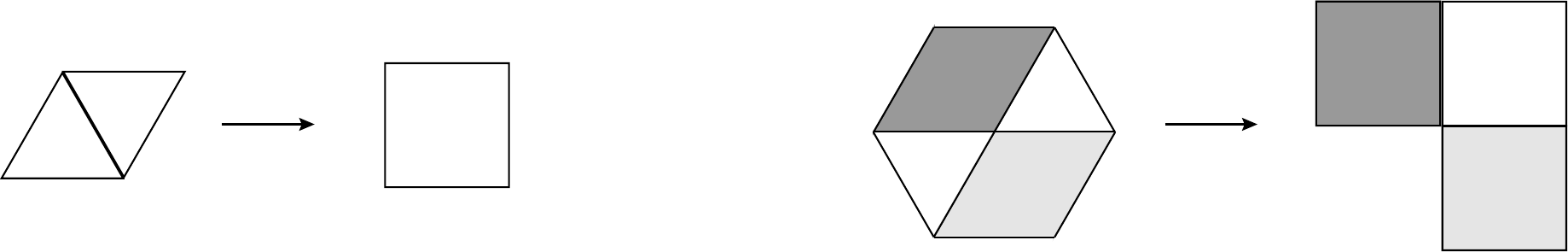}
    \caption{Shear maps on the doubled triangle and hexagon.} 
    \label{fig: shear}
\end{figure}

In this paper, we compute the Kontsevich--Zorich monodromy group for three translation surfaces: the translation cover (3-cover) of the octahedron, the 4-cover of the cube, and the 2-cover of the mutetrahedron. As noted in the Remark~\ref{rem: common covers}, the octahedron and Octa-8 have common translation covers, and so do the cube and Mucube. The translation cover of the tetrahedron is (conjugate to) a square torus. In particular, for the tetrahedron, the absolute homology is the same as the tautological plane and hence the Kontsevich--Zorich monodromy group is trivial. 

We choose these examples because the genus of the unfoldings are relatively low (less than 10.) The computation can be carried out on the unfolding of other surfaces (except the dodecahedron) in Table~\ref{tab: platonic surfaces}. However, in this paper we do not carry them out for practical reasons.

\subsection{Monodromy constraints}
Let $G = \Aut(X)$ denote the automorphism group. By Hurwitz's theorem on the order of automorphism groups, $G$ is a finite group. The vector space $H_1(X;\R)$ has the structure of a $G$-module induced by a representation $\rho\colon \Aut(X)\to \Sp(H_1(X;\R))$ arising from the homological action of $G$ on $H_1(X;\R)$. Since $G$ is finite, we now are able to utilize the well understood theory of finite dimensional groups (see e.g.\ Fulton and Harris \cite{MR1153249} or Serre \cite {MR0232867}). In particular, we have finitely many irreducible representations which we denote $\text{Irr}_\mathbb R (G)$. We decompose $H_1(S;\mathbb R)$ into irreducible pieces,
$$H_1(S;\mathbb R)=\bigoplus_{\tau\in \text{Irr}_\mathbb R (G)}V_\tau ^{ n_\tau}$$
where each $V_\tau$ is an irreducible subspace of  $H_1(S;\mathbb R)$ on which $G$ acts as the representation $\tau.$ The subspaces $W_\tau \colonequals V_\tau ^{ n_\tau}$ of the same $G$-irreducible representations are called \emph{isotypical components}. That is, $W_\tau$ is $n_\tau$ copies of the vector space $V_\tau$ corresponding to the irreducible representation $\tau\colon G\to \Sp(V_\tau)$ which is equal to $\rho|_{V_\tau}$.

It is worth noting that $\Aff(X)$ may not respect the decomposition of $H_1(S; \R)$ into a direct sum of isotypical components, since isotypical components are not, in general, $\Aff(X)$-invariant. Nevertheless, it is possible to pass from $\Aff(X)$ to a finite-index subgroup which does respect this decomposition. Observe that if $R \in \Aff(X)$ and $\pi \in G$, then $R \pi R^{-1} \in G$. Thus, we obtain a group homomorphism $\Aff(X) \to \Sym(G)$. The index of the kernel $\widetilde{\Aff}(X)$ of this homomorphism is finite inside $\Aff(X)$, since $\Sym(G)$ is finite. Since $\widetilde{\Aff}(X)$ consists exactly of the elements of $\Aff(X)$ that commute with each element of $G$, if $V_\tau$ is an irreducible subspace of $H_1(S; \R)$ on which $G$ acts as the representation $\tau$, and $R \in \widetilde{\Aff}(X)$, then $R V_\tau$ is also a subspace where $G$ acts as the representation $\tau$. Hence, the action of $\widetilde{\Aff}(X)$ may permute the different pieces that make up the isotypical component $W_\tau$, but it preserves $W_\tau$. Since we are mainly interested in computations regarding Zariski closures and since finite-index subgroups do not change the Zariski closure, we will ignore the need to pass from $\Aff(X)$ to a finite-index subgroup.

For each irreducible representation $\tau$ of $G$, we consider the associative division algebra, $D_\tau$, given by the centralizer of $\tau(G)$, the image of $\tau$, inside of $\text{End}_{\mathbb R} (V_\tau)$. That is,
$$D_\tau = \{X\in M_{\dim_\R {V_\tau}}(\R):\tau(G)X=X\tau(G)\}.$$


In Matheus--Yoccoz--Zmiaikou \cite{MYZ}, the authors show that the type of division algebra constrains the Zariski closure of the monodromy group. 

\begin{prop}[Proposition 3.16 of \cite{MYZ}]\label{PropMYZ} Let $W_\tau \colonequals V_\tau ^{n_\tau}$ be an isotypical component of the action of $\Aut(X)$ on $H_1(X;\R)$. If $D_\tau \simeq \mathbb R$, then $$\overline{\tilde \alpha(\Aff(X))|_{W_\tau}}^{\Zar}\subseteq \Sp(n_\tau,\R).$$
\end{prop}

In fact, Proposition 3.16 of \cite{MYZ} says much more depending on the type of division algebra, but in this article we only obtain real division algebras. This result allows us to give an ``upper bound" on the Kontsevich--Zorich monodromy groups.

Additionally we use the following result of \cite{MYZ} that allows us to gain information about the Lyapunov exponents from the isotypical components.

\begin{prop}\label{prop_MYZ_exp}
Let  $W_\tau=V_{\tau} ^{n_\tau}$ be an isotypical component.
The multiplicity in $W_\tau$ of each Lyapunov exponent is a multiple of the dimension (over $\mathbb R$) of $V_\tau.$

In particular, if $n_\tau$ is 2, then the multiplicity of a Lyapunov exponent in $W_\tau$ is the maximum possible, i.e.\ the dimension of $V_\tau.$
\end{prop}

The previous result may be informally explained as follows. As we previously discussed, an isotypical component $W_\tau$ is preserved by a finite-index subgroup $\widetilde{\Aff}(X)$ of $\Aff(X)$, but the different subspaces that make up $W_\tau$ may be permuted by this action. This means that isotypical components are, in general, not strongly irreducible. Nevertheless, since $\widetilde{\Aff}(X)$ permutes the irreducible representations inside $W_\tau$, the Lyapunov spectrum gets carried from one irreducible representation inside $W_\tau$ to another. Thus, the Lyapunov spectrum inside $W_\tau$ ends up consisting of copies of the same Lyapunov spectrum multiple times.
 

\subsection{The Eskin--Kontsevich--Zorich formula} While computing individual Lyapunov exponents is, in general, extremely hard, it is known that the \emph{sum} of positive Lyapunov exponents is a rational number that can be explicitly found for the case of square-tiled surfaces \cite[Corollary 8]{MR3270590}. This formula is stated in terms of widths and heights of the horizontal cylinders that make up the square-tiled surface $X$ and the elements of $\SL(2, \mathbb{Z}) \cdot X$. While we will not explicitly state the formula here, we stress that this number can be quickly computed using \texttt{surface\_dynamics}.

\subsection{Notations} In order to compute the Zariski closure of the Kontsevich--Zorich monodromy groups, we will need the automorphism group and the affine diffeomorphism group of a square-tiled surface. We will denote the generators of the automorphism groups by $\pi_i$ and the generators of the affine diffeomorphism group by products of $T=\begin{pmatrix} 1 & 1 \\ 0 & 1\end{pmatrix}$ and $S=\begin{pmatrix} 1 & 0 \\ 1 & 1\end{pmatrix}.$ We will denote the representation of the automorphism group on the homology by $\rho$ and affine diffeomorphism by $\tilde{\alpha},$ the representation of the affine diffeomorphisms on the zero-holonomy subspace by $\alpha.$

\section{On the 3-cover of the octahedron}\label{sec: octa}
In this section, we study the threefold translation cover of the octahedron as a square-tiled surface, and compute the Kontsevich--Zorich monodromy group of the translation surface.

\subsection{Unfolding of the octahedron as a square-tiled surface}

The regular octahedron has cone angle $\frac{4\pi}{3}$ at each vertex. Its unfolding 3-fold cover is a translation surface of genus four. We apply appropriate shear and dilation maps (Figure~\ref{fig: shear}) and consider its associated square-tiled surface along with a basis of homology. The squares are indexed by numbers from 1 to 24, and the squares above and to the right of square $i$ are denoted by $\sigma_v(i)$ and $\sigma_h(i)$ respectively.

In this section, we study the surface $O$, which is the image under $T^2$ of the translation cover from\cite{AAH}. The monodromy group is the same for any surface in the orbit. By the \texttt{surface\_dynamics} package, we get $$\SL(O)=\langle T^3, S^{-1}T \rangle < \SL(2,\Z).$$ We will use these generators to compute the Kontsevich--Zorich monodromy group of the translation cover of the octahedron. 

\begin{figure}[htbp] 
	\centering
    \includegraphics[width=6in]{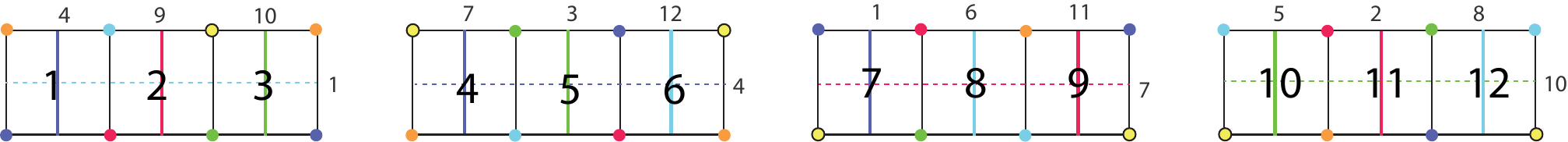}
    \caption{$O$ along with a basis of homology} 
    \label{fig:M4}
\end{figure}

\subsection{Basis of homology and zero-holonomy}
The absolute homology of $O$ is 8-di\-men\-sion\-al and an explicit basis can be given by the horizontal curves of $O$ that we denote as $\sigma_i$ where $i=1,4,7,10$, that begin on the left side of square $i$ along with the vertical curves of $O$ that we denote as $\zeta_j$ where $j=1,2,3,6$, that begin on the bottom side of square $j$ (Figure~\ref{fig:M4}). The holonomy vectors of $\sigma_i$ and $\zeta_j$ on $O$ are $\begin{pmatrix}3 \\ 0\end{pmatrix}$ and $\begin{pmatrix} 0 \\ 3\end{pmatrix},$ respectively.

With $\Sigma_i  \colonequals \sigma_i  - \sigma_{10}$ and $Z_j \colonequals \zeta_1  - \zeta_j,$ we define $\{ \Sigma_i, Z_j \}$ as the basis of the zero-holonomy subspace of the homology.

\subsection{Intersection form} 
We record the intersection matrix encoding the algebraic intersection form of $O$ with respect to the basis  given above. The intersection matrix of $O,$ which we denote by $\Omega$ is given by 
	\[
		\Omega= \left(\begin{array}{c|c} 
			 \mbox{\Large $0$} & \begin{matrix}
				1 & 1 & 1 & 0  \\
				1 & 0 & 1 & 1\\
				1 & 1 & 0 & 1 \\
				0 & 1 & 1 & 1  \\
			\end{matrix} \\[0.5ex] & \\[-5.5ex] & \\[0.5ex] \hline & \\[-2ex]
			 \begin{matrix}
				-1 & -1 & -1 & 0  \\
				-1 & 0 & -1 & -1\\
				-1 & -1 & 0 & -1 \\
				0 & -1 & -1 & -1  \\
			\end{matrix} & \mbox{\Large $0$} 
		\end{array}\right)
	.\]
	
\subsection{Action of the automorphism group on homology}\label{subsec: action of aut}

By \texttt{surface\_dynamics}, we get $\Aut(O)\simeq A_4,$ the alternating group on 4 elements, that is generated by the permutations 
\begin{align*}
\pi_1 & = (1,3,2)(4,10,9)(5,11,7)(6,12,8),\\
\pi_2 & = (1,6,11)(2,4,12)(3,5,10)(7,8,9).
\end{align*}

Figure~\ref{fig:pi1 on M4} describes the action of $\pi_1$ on $O$:
\begin{figure}[htbp] 
	\centering
    \includegraphics[width=4.5in]{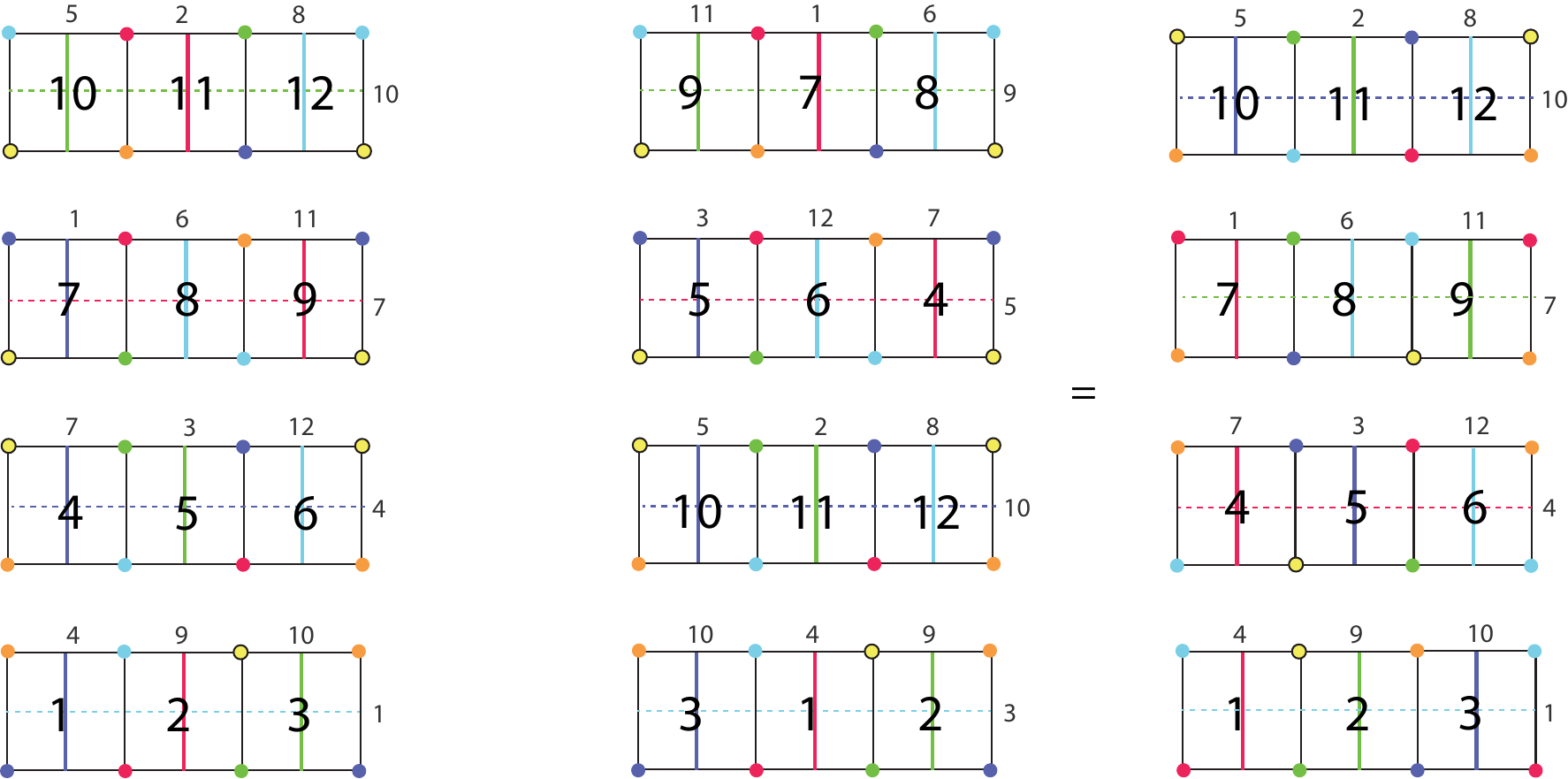}
    \caption{$O$ (left), action of $\pi_1$ on $O$ (center), $\pi_1(O)$ after cut-and-paste (right).}
    \label{fig:pi1 on M4}
\end{figure}

By observation, see Figure~\ref{fig:pi1 on M4}, we have
\begin{alignat*}{2}
\sigma_1&\mapsto \sigma_1, \qquad & \zeta_1&\mapsto\zeta_3,\\
\sigma_4&\mapsto \sigma_{10}, \qquad & \zeta_2&\mapsto\zeta_1,\\
\sigma_7&\mapsto \sigma_4, \qquad & \zeta_3&\mapsto\zeta_2,\\
\sigma_{10}&\mapsto \sigma_7, \qquad & \zeta_6&\mapsto\zeta_6.
\end{alignat*}

Denote by $\rho$ the representation arising from the action of $\Aut(O)$ on $H_1(O;\R).$ Then the action of $\pi_1$ on $H_1(O;\R)$ with respect to the ordered basis $\{\sigma_1,\sigma_4,\sigma_7,\sigma_{10},\zeta_1,\zeta_2,\zeta_3,\zeta_6\}$ is given by
	\[
		\rho(\pi_1)  = \left(\begin{array}{c|c} 
			\begin{matrix}
				1 & 0 & 0 & 0   \\
				0 & 0 & 1 & 0 \\
				0 & 0 & 0 & 1 \\
			       0 & 1 & 0 & 0 \\
			\end{matrix}  & \mbox{\Large $0$} \\[0.5ex] & \\[-5.5ex] & \\[0.5ex] \hline & \\[-2ex]
			 \mbox{\Large $0$}  & 			\begin{matrix}
				0 & 1 & 0 & 0   \\
				0 & 0 & 1 & 0 \\		
				1 & 0 & 0 & 0 \\
			       0 & 0 & 0 & 1 \\
			\end{matrix}  
		\end{array}\right).
	\]
A similar computation shows the action of  $\pi_2$ on $H_1(O;\R)$ :
	\[
		\rho(\pi_2)  = \left(\begin{array}{c|c} 
			\begin{matrix}
			0 & 0 & 0 & 1   \\
			1 & 0 & 0 & 0 \\
			0 & 0 & 1 & 0 \\
			 0 & 1 & 0 & 0 \\
			\end{matrix}  & \mbox{\Large $0$} \\[0.5ex] & \\[-5.5ex] & \\[0.5ex] \hline & \\[-2ex]
			 \mbox{\Large $0$}  & 
			 \begin{matrix}
			0 & 1 & 0 & 0 \\
			0 & 0 & 0 & 1 \\
			0 & 0 & 1 & 0 \\		
			1 & 0 & 0 & 0 \\
			\end{matrix}  
		\end{array}\right).
	\]

\subsection{Action of the affine group on homology and monodromy of the 3-cover of the octahedron}
In this section, we compute the action of $\Aff(O)$ on the absolute homology of $O$. As a corollary of our computations, we obtain generators for the Kontesevich--Zorich monodromy group of $O$.

Recall $\tilde \alpha\colon \Aff(O)\to \Sp(8,\mathbb R)$ denotes the representation arising from the action of the affine diffeomorphisms on $O$. In what follows, we actually compute the action of the Veech group and note that all the calculations and matrices only make sense up to the action of $\Aut(O)$. Let $\alpha\colon \Aff(O)\to \Sp(6,\mathbb R)$ denote the action on the zero-holonomy subspace.

The main result in this section is the following:

\begin{thm}
The Kontsevich--Zorich monodromy group of $O$ is generated by the following two matrices
	\[
		\alpha(T^3)  = \left(\begin{array}{c|c} 
			 \mbox{\Large $\mathrm{Id}_3$} & \begin{matrix}
				0 & 0 &  1   \\
				1 & 0 & 0 \\
				0 & 1 &  0 \\
			\end{matrix}  \\[0.5ex] & \\[-5.5ex] & \\[0.5ex] \hline & \\[-2ex]
		\mbox{\Large $0$} & 	\mbox{\Large $\mathrm{Id}_3$}
		\end{array}\right) \qquad \text{and} \qquad 
\alpha(S^{-1}T)  = \left(\begin{array}{c|c} 
			 \begin{matrix}
				-1 & -1 &  -1   \\
				0 & 0 & 1 \\
				1 & 0 &  0 \\
			\end{matrix} & \begin{matrix}
				1 & 0 &  1   \\
				0 & 0 & 0 \\
				0 & 0 &  -1 \\
			\end{matrix}  \\[0.5ex] & \\[-5.5ex] & \\[0.5ex] \hline & \\[-2ex]
		\begin{matrix}
				0 & 0 & 1   \\
				1 & 0 & 0 \\
				-1 & 0 &  -1 \\
			\end{matrix} & 	\mbox{\Large $0$}
		\end{array}\right).
	\]
\end{thm}

\begin{proof}
It suffices to compute the representation $\tilde\alpha\colon \Aff(O) \to \Sp(8,\R)$ and restrict to a basis of the zero-holonomy subspace. We begin by computing the generators $\tilde\alpha(T^3)$  and $\tilde\alpha(S^{-1}T)$ because $T^3$ and $S^{-1}T$ generate the Veech group of $O$.

Figure~\ref{fig:T^3(M4)} describes the action of $T^3$ on $O.$ Since $T^3$ is a horizontal shear map, all horizontal curves are mapped to themselves. 

\begin{figure}[htbp] 
	\centering
    \includegraphics[width=3in]{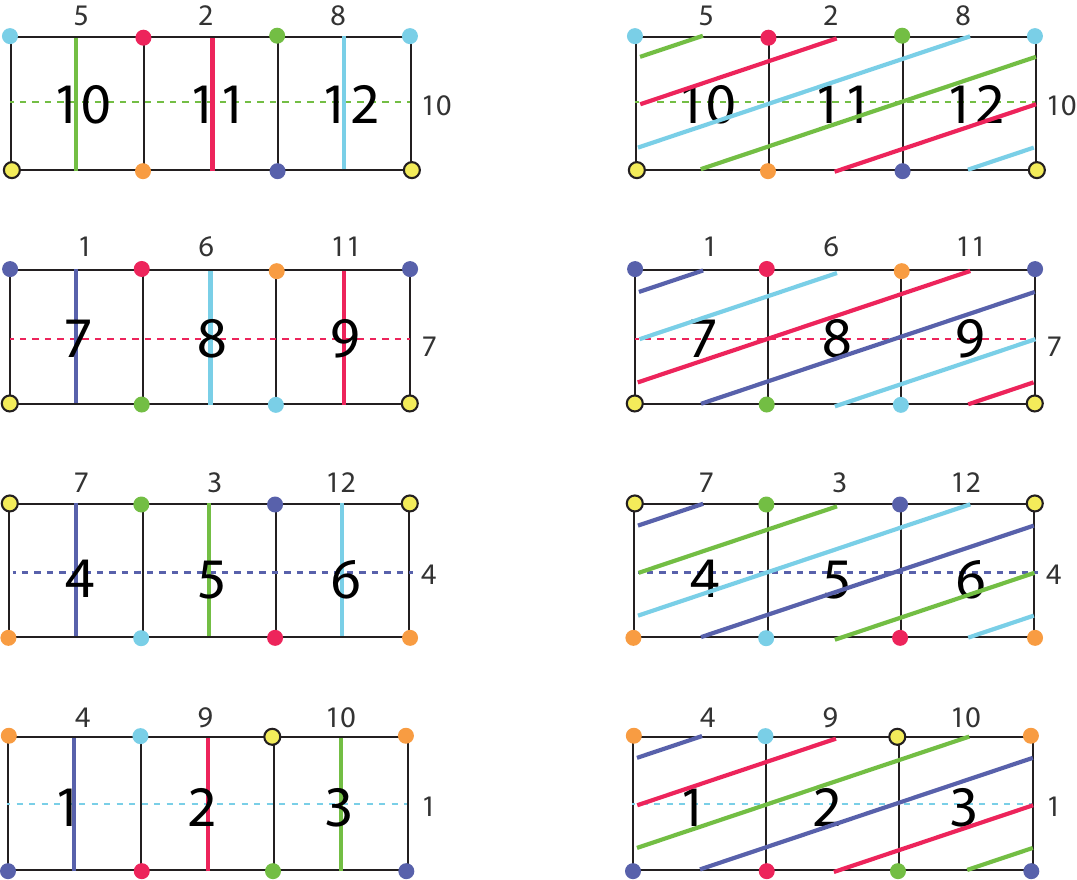}
    \caption{$O$ (left) and $T^3(O)$ (right)}
    \label{fig:T^3(M4)}
\end{figure}

To compute the image of the vertical curves under $T^3,$ we take advantage of the intersection matrix $\Omega.$ For example, write  $\gamma \colonequals \tilde \alpha(T^3)(\zeta_1)$ as a linear combination of the basis vectors. That is $\gamma = \sum a_i \sigma_i + b_j \zeta_j$ for some $a_i$ and $b_j.$  Then, $\langle \gamma,\sigma_1\rangle =\sum_j \langle \zeta_j,\sigma_1\rangle = -b_1-b_2-b_3.$ Thus, by counting the intersection number $ \langle \gamma,\sigma_1\rangle,$ we obtain equation on $a_i, b_j$ that determine $\gamma$. By counting the intersection number between $\gamma$ and all other curves, we obtain the following system of equations:

\begin{align*}
   -1 = \langle \gamma,\sigma_1\rangle &=-b_1-b_2-b_3\\
   -1 = \langle \gamma,\sigma_4\rangle &=-b_1-b_3-b_6\\
   -1 = \langle \gamma,\sigma_7\rangle &=-b_1-b_2-b_6\\
   0 = \langle \gamma,\sigma_{10}\rangle &=-b_2-b_3-b_6\\
    3 = \langle \gamma,\zeta_1\rangle &=a_1+a_4+a_7\\
    2 = \langle \gamma,\zeta_2\rangle &=a_1+a_7+a_{10}\\
    2 = \langle \gamma,\zeta_3\rangle &=a_1+a_4+a_{10}\\
    2 = \langle \gamma,\zeta_6\rangle &=a_4+a_7+a_{10}
\end{align*}

Solving this system yields $\tilde \alpha(T^3)(\zeta_1) = \sigma_1+\sigma_4+\sigma_7+\zeta_1.$

We repeat this for the rest of the curves and obtain the following:
\begin{align*}
\tilde \alpha(T^3) (\sigma_i) &= \sigma_i, \, i=1,4,7,10,\\
\tilde \alpha(T^3)(\zeta_1) &= \sigma_1+\sigma_4+\sigma_7+\zeta_1,\\
\tilde \alpha(T^3)(\zeta_2) &= \sigma_1+\sigma_7+\sigma_{10}+\zeta_2,\\
\tilde \alpha(T^3)(\zeta_3) &= \sigma_1+\sigma_4+\sigma_{10}+\zeta_3,\\
\tilde \alpha(T^3)(\zeta_6) &= \sigma_4+\sigma_7+\sigma_{10}+\zeta_6.
\end{align*}

We use the above to compute $\alpha\colon \Aff(O) \to \Sp(6,\R),$ the action of the affine group on the zero-holonomy subspace. We define the basis of the subspace by $\Sigma_i  \colonequals \sigma_i  - \sigma_{10}$ for $i=1,4,7$ and $Z_j \colonequals \zeta_1  - \zeta_j $ for $j=2,3,6$.

Then, by definition of $\Sigma_i$ and the action of $T^3$ on $\sigma_i$, we have,
\begin{align*}
\alpha(T^3)(\Sigma_i)&=\tilde\alpha(T^3)(\sigma_i-\sigma_{10})=\tilde\alpha(T^3)(\sigma_i)-\tilde\alpha(T^3)(\sigma_{10})\\
&=\sigma_i-\sigma_{10}=\Sigma_i.
\end{align*}

For $Z_j$, $j=2,3,6$, take $Z_2$ for example, we obtain
\begin{align*}
\alpha(T^3)(Z_2)&=\tilde\alpha(T^3)(\zeta_1-\zeta_2)\\
&=\tilde\alpha(T^3)(\zeta_1)-\tilde\alpha(T^3)(\zeta_2)\\
&=(\zeta_1+\sigma_1+\sigma_4+\sigma_7)-(\zeta_2+\sigma_1+\sigma_7+\sigma_{10})\\
&=\sigma_4-\sigma_{10}+\zeta_1 -\zeta_2\\
&=\Sigma_4+Z_2.
\end{align*}
Similarly,  $\alpha(T^3)(Z_3)=\Sigma_7+Z_3$ and $\alpha(T^3)(Z_6)=\Sigma_1 +Z_6$.

Hence, with respect to the ordered basis $\{\Sigma_1,\Sigma_4,\Sigma_7,Z_2,Z_3,Z_6\}$ we have the matrix $\alpha(T^3)$ as stated in the theorem.

A similar idea shows the action of $S^{-1}T$ acts on the basis of homology as:
\begin{align*}
\tilde\alpha(S^{-1}T)(\sigma_1) &= \sigma_7-\zeta_3,\\  
\tilde\alpha(S^{-1}T)(\sigma_4) &= \sigma_{10}-\zeta_1,\\   
\tilde\alpha(S^{-1}T)(\sigma_7) &= \sigma_4-\zeta_2,\\   
\tilde\alpha(S^{-1}T)(\sigma_{10}) &= \sigma_1-\zeta_6,\\   
\tilde\alpha(S^{-1}T)(\zeta_1) &= \sigma_1,\\
\tilde\alpha(S^{-1}T)(\zeta_2) &= \sigma_{10},\\
\tilde\alpha(S^{-1}T)(\zeta_3) &= \sigma_4,\\
\tilde\alpha(S^{-1}T)(\zeta_6) &= \sigma_7.
\end{align*}

Using the above, we compute the action of $S^{-1}T$ on the zero-holonomy subspace as we do for $T^3$.
\end{proof}

\subsection{Identification of the Kontsevich--Zorich monodromy group}
In this section, we use the representations arising from the automorphisms $\rho$ and affine diffeomorphisms $\alpha$ to identify the Kontsevich--Zorich monodromy group.

 \begin{thm}
 The Zariski closure of the monodromy group of the translation cover of the octahedron is
 $$\overline{ \alpha(\Aff(O))} ^{\Zar}\simeq \Sp(2,\R).$$
 \end{thm}
 
 \begin{proof} 
We  begin by computing the isotypical components of $H_1(O;\R)$. Let $$V_1 = \text{span}_\R\{\sigma_i ^4\}_{i=1,4,7,10} \text{ and } 
V_2 = \text{span}_\R\{\zeta_j ^4\}_{j=1,2,3,6}.$$

Notice that $\rho(\Aut(O))$ preserves each $V_i$, $i=1,2$.  We focus on decomposing $V_1$ and note that analogous statements hold for $V_2$.

One can see that $\sum_i\sigma_i$ is as a simultaneous eigenvector for $\rho(\Aut(O))$ with eigenvalue 1 because it is an eigenvector for the generators of $\rho(\Aut(O))$ with eigenvalue 1. Let $E_1$ denote the span of this eigenvector. Then the orthogonal complement of $E_1$ inside of $V_1$ is a 3-dimensional $\rho(\Aut(O))$-invariant subspace of $V_1$ that we denote by $Z_1$. A basis is given by $Z_1 = \text{span}_\R\{\sigma_i-\sigma_{10}\}_{i=1,4,7}$.  Thus, $V_1 = E_1\oplus Z_1$.

Similarly, $V_2 = E_2\oplus Z_2$, where $E_2$ is the 1-dimensional simultaneous eigenspace given by the span of $\sum_j\zeta_j$ and $Z_2$ is the orthogonal complement of $E_2$ in $V_2.$ Note that the direct sum of the 1-dimensional subspaces is simply the tautological plane and so the orthogonal complement is given by the zero-holonomy subspace. Hence, the isotypical components of $H_1(O;\R)$ are given by $H_1^{\mathrm{st}}(O) = E_1\oplus E_2$ and $H_1^{(0)}(O) = Z_1\oplus Z_2.$ 

Since $H_1^{\mathrm{st}}(O) $ is the tautological subspace, the action of the monodromy group here is given by $\Sp(2,\R)$. It remains to compute $\tilde \alpha(\Aff(O))|_{H_1^{(0)}(O)} =  \alpha(\Aff(O))$.

To use the results of \cite{MYZ} we need to compute the centralizer of $\rho(\Aut(O))|_{Z_1}.$ Recall the matrices from Section~\ref{subsec: action of aut} and we get  \[
\rho(\pi_1)|_{Z_1}=
\left(
\begin{array}{ccc}
 1 & 0 & 0 \\
 0 & 0 & 1 \\
 -1 & -1 & -1 \\
\end{array}
\right)\text{ and }
\rho(\pi_2)|_{Z_1}=
\left(
\begin{array}{ccc}
 -1 & -1 & -1 \\
 1 & 0 & 0 \\
 0 & 0 & 1 \\
\end{array}
\right)
\]

Solving the system 
\begin{align*}
    \rho(\pi_1)|_{Z_1}X&=X\rho(\pi_1)|_{Z_1}\\
    \rho(\pi_2)|_{Z_1}X&=X\rho(\pi_2)|_{Z_1}
\end{align*} 
yields that $X$ is of the form $a\mathrm{Id}_3$ and so the centralizer of $\rho(\Aut(O))|_{Z_1}$ is  $\{a\mathrm{Id}_3:a\in \R\}\simeq \R$. Thus, this representation is real and by Proposition \ref{PropMYZ} we have $\overline{ \alpha(\Aff(O))} ^{\Zar}< \Sp(2,\R)$. This gives an ``upper bound" for the monodromy group.

Now we compute a ``lower bound." Since $\alpha(T^3)$ is parabolic, the logarithm $t = \log(\alpha(T^3))$ is in the Lie algebra of the algebraic group $\overline{ \alpha(\Aff(O))} ^{\Zar}$. 

Let $\phi_X$ denote the conjugation map $\phi_X(g)=XgX^{-1}$. A direct computation shows that the following three elements are linearly independent inside of the Lie algebra of $\overline{ \alpha(\Aff(O))} ^{\Zar},$
$$t, \phi_{\alpha(S^{-1}T)}(t),\phi_{(\alpha(S^{-1}T)^2)}(t).$$

Since $\overline{ \alpha(\Aff(O))} ^{\Zar}$ sits inside the 3-dimensional group $\Sp(2,\R)$ and the dimension has a lower bound of 3, we conclude the Zariski closure of the monodromy group of the translation cover of the octahedron is
 $$\overline{ \alpha (\Aff(O))} ^{\Zar}\simeq \Sp(2,\R).$$
 \end{proof}

 \subsection{Lyapunov exponents of the 3-cover of the octahedron}

We compute the Lyapunov spectrum of the 3-cover of the octahedron $O$.
\begin{prop}
Counting multiplicities, the positive Lyapunov spectrum of the translation cover of the octahedron is
$$\{1\} \cup \{1/2,1/2,1/2\},$$
where the union indicates the exponents corresponding to distinct, symplectically-orthogonal, irreducible pieces of the Hodge bundle.
\end{prop}
\begin{proof}
Recall we have the following decomposition of the homology of $O$ into isotypical components,  $$(E_1\oplus E_2)\oplus(Z_1\oplus Z_2).$$ 
 The space $E_1\oplus E_2$ is the tautological plane and carries a Lyapunov exponent of 1. The Eskin--Kontsevich--Zorich formula \cite{MR3270590} and the \texttt{surface\_dynamics} package yield that the sum of the Lyapunov exponents on $O$ must be 5/2. On the other hand, Proposition \ref{prop_MYZ_exp}, tells us that the multiplicity of the Lyapunov exponent in the remaining isotypical component $Z_1\oplus Z_2$ is 3. Putting these together yields that the remaining Lyapunov exponent is 1/2 with a multiplicity of 3.
\end{proof}


\section{On the 4-cover of the cube}\label{sec: cube}

\subsection{Unfolding of the cube}
The two permutations that define the 4-cover of the cube are given by 

\begin{align*} h_1 = (1,2,3,4)(5,6,7,8)(9,10,11,12)(13,14,15,16)(17,18,19,20)(21,22,23,24)\\
v_1 = (1,9,14,22)(2,20,13,7)(3,24,16,11)(4,5,15,18)(6,10,17,21)(8,23,19,12)\end{align*}

We call this genus-nine surface $C_1,$ whose Veech group is an index-9 subgroup of $\SL(2,\Z)$.

In this section, we study $C \colonequals T S^2 C_1$ (see Figure~\ref{fig:pi1(C2)}).

By the \texttt{surface\_dynamics} package, we have
$$\SL(C)=\left\langle \begin{pmatrix} 1 & 2 \\ 0 & 1\end{pmatrix}, \, \begin{pmatrix} 5 & -2 \\ 3 & -1\end{pmatrix}, \, \begin{pmatrix} 3 & -2 \\ 5 & -3\end{pmatrix}\right\rangle <\SL(2,\Z).$$ 


\subsection{Basis of homology and zero-holonomy}
For $C,$ we use the homology basis with horizontal curves $\gamma,$ vertical curves $\zeta,$ and slope 1 curves $\eta$ (see Figure \ref{fig:pi1(C2)}): $$\{\sigma_1, \sigma_2, \sigma_3, \sigma_5, \sigma_8, \sigma_9, \zeta_1, \zeta_2, \zeta_3, \zeta_4, \zeta_5, \zeta_8, \eta_1, \eta_2, \eta_3, \eta_4, \eta_6, \eta_8\}.$$

The holonomy of $\sigma_i$ are $\begin{pmatrix}2\\
0\end{pmatrix},$ $\zeta_j$ are $\begin{pmatrix}0\\
3\end{pmatrix},$ and $\eta_k$ are $\begin{pmatrix}4\\
4\end{pmatrix}.$ For the zero-holonomy subspace, we will use the following basis: \begin{alignat*}{2}
    \Sigma_i &\colonequals \sigma_i - \sigma_9, \qquad & i&=1,2,3,5,8;\\
    Z_j &\colonequals \zeta_j - \zeta_8, \qquad & j&=1,2,3,4,5;\\
    H_k &\colonequals \eta_k-2\sigma_9-\frac{4}{3}\zeta_8, \qquad & k&=1,2,3,4,6,8.
\end{alignat*}

\subsection{Intersection form}

We record the intersection matrix encoding the algebraic intersection form of $C$ with respect to the basis given above. The intersection matrix of $C$  is given by$$\begin{pmatrix} 0 & -A & -B\\
A^T & 0 & -C\\
B^T & C^T & 0\end{pmatrix}$$ where

$$A = \begin{pmatrix}1 & 0 & 0 & 0 & 0 & 0\\
0 & 1 & 0 & 0 & 0 & 0\\
0 & 0 & 1 & 0 & 0 & 1\\
0 & 0 & 0 & 0 & 1 & 0\\
0 & 0 & 0 & 0 & 0 & 1\\
0 & 0 & 0 & 0 & 0 & 0\end{pmatrix}, \qquad B = \begin{pmatrix}1 & 0 & 0 & 1 & 0 & 0\\
1 & 1 & 0 & 0 & 0 & 0\\
0 & 1 & 1 & 0 & 0 & 0\\
0 & 0 & 0 & 1 & 0 & 1\\
0 & 1 & 0 & 0 & 0 & 1\\
1 & 0 & 0 & 0 & 0 & 1\end{pmatrix},$$
and
$$ C = \begin{pmatrix} -1 & 0 & 0 & -1 & -1 & 0\\
-1 & -1 & 0 & 0 & -1 & 0\\
0 & -1 & -1 & 0 & -1 & 0\\
0 & 0 & -1 & -1 & -1 & 0\\
0 & 0 & -1 & -1 & 0 & -1\\
0 & -1 & -1 & 0 & 0 & -1\end{pmatrix}.$$

\subsection{Action on homology and monodromy of the 4-cover of the cube} \label{sec: cube aut}

The aim of this section is to compute generators for the action arising from the automorphism group.

Using \texttt{surface\_dynamics} we find that the automorphism group of $C$ is the subgroup of $S_{24}$ generated by $$\begin{array}{rl}\pi_1 & = (1,13)(2,14)(3,15)(4,16)(5,11)(6,12)(7,9)(8,10)(17,23)(18,24)(19,21)(20,22),\\
\pi_2 & = (1,14)(2,15)(3,16)(4,13)(5,7)(6,8)(9,22)(10,23)(11,24)(12,21)(17,19)(18,20),\\
\pi_3 & = (1,23,5)(2,24,6)(3,21,7)(4,22,8)(9,19,15)(10,20,16)(11,17,13)(12,18,14).\end{array}$$

In fact, $\Aut(C)$ is isomorphic to $S_4$.

\begin{figure}[htbp] 
	\centering
    \includegraphics[width=6in]{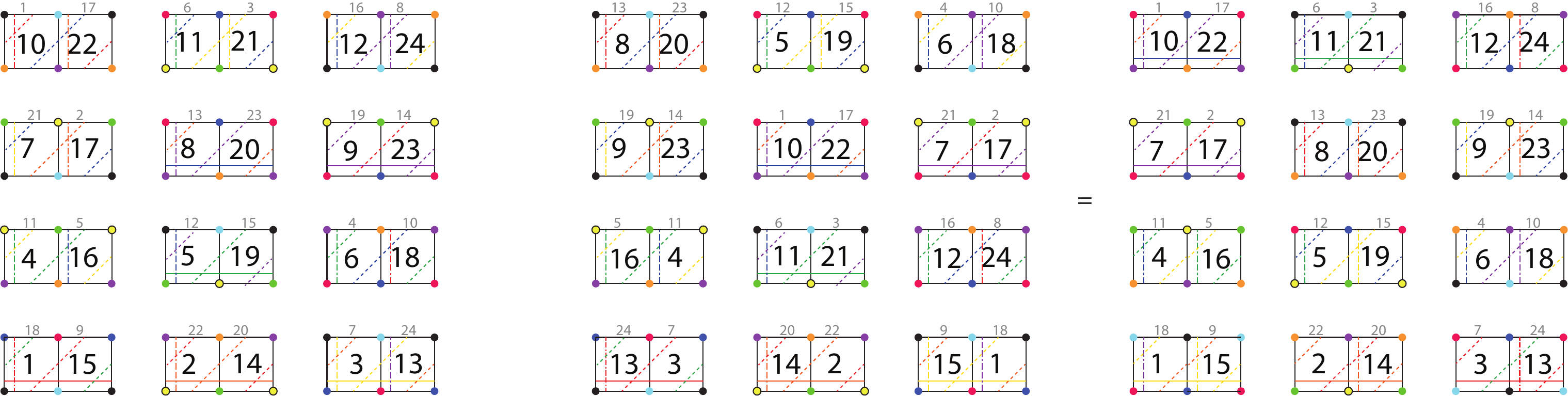}
    \caption{$C$ (left), action of $\pi_1$ on $C$ (center), $\pi_1(C)$ after cut-and-paste (right)}
    \label{fig:pi1(C2)}
\end{figure}

Figure~\ref{fig:pi1(C2)} shows $C$ and the action of $\pi_1$ on $C$ (opposite vertical sides are identified). By observation we see
\begin{align*}
    \sigma_1&\mapsto\sigma_3, & \sigma_2&\mapsto\sigma_2, & \sigma_3&\mapsto \sigma_1, & \zeta_1&\mapsto\zeta_8, & \zeta_4&\mapsto\zeta_5, & \zeta_5&\mapsto\zeta_4, & \zeta_8&\mapsto\zeta_1,\\
    \eta_1&\mapsto\eta_2, & \eta_2&\mapsto\eta_1, & \eta_3&\mapsto\eta_4, & \eta_4&\mapsto\eta_3, & \eta_6&\mapsto\eta_8, & \eta_8&\mapsto\eta_6.
\end{align*}

Using the intersection form for the remaining curves, we obtain
\begin{alignat*}{2}
\sigma_5&\mapsto{} & &{-2\sigma_1}+2\sigma_3+\sigma_9 -\frac{3}{2}\zeta_1 + \frac{1}{2}\zeta_2 -\frac{1}{2}\zeta_5+\frac{3}{2}\zeta_8+\frac{1}{2}\eta_1-\eta_2-\frac{1}{2}\eta_3+\eta_4+\frac{1}{2}\eta_6-\frac{1}{2}\eta_8,\\
\sigma_8&\mapsto{} & &4\sigma_1-\sigma_2-5\sigma_3-\sigma_5-\sigma_8-\sigma_9+3\zeta_1-2\zeta_2,-\zeta_3-\zeta_4+\zeta_5-4\zeta_8-{}\\
&&&\eta_1+3\eta_2+2\eta_3-2\eta_4+\eta_8,\\
\sigma_9&\mapsto{} & &{-\sigma_1}+\sigma_3+\sigma_5-\frac{1}{2}\zeta_1-\frac{1}{2}\zeta_2+\frac{1}{2}\zeta_5+\frac{1}{2}\zeta_8+\frac{1}{2}\eta_1-\frac{1}{2}\eta_3+\frac{1}{2}\eta_6-\frac{1}{2}\eta_8,\\
\zeta_2&\mapsto{} & &{-2\sigma_1}+2\sigma_3-2\zeta_1+\zeta_2+\zeta_4-\zeta_5+2\zeta_8+\eta_1-\eta_2-\eta_3+\eta_4,\\
\zeta_4&\mapsto{} & &{-2\sigma_1}+2\sigma_3-\zeta_1+\zeta_3+\zeta_8+\eta_1-\eta_2-\eta_3+\eta_4.
\end{alignat*}

Let $\rho$ denote the representation arising from the action of $\Aut(C)$ on $H_1(C;\R)\simeq Sp(18,\mathbb R)$. Then the action of $\pi_1$ on $H_1(C;\R)$ with respect to the ordered basis $$\{\sigma_1, \sigma_2, \sigma_3, \sigma_5, \sigma_8, \sigma_9, \zeta_1, \zeta_2, \zeta_3, \zeta_4, \zeta_5, \zeta_8, \eta_1, \eta_2, \eta_3, \eta_4, \eta_6, \eta_8\}$$ is given by
	\[
		\rho(\pi_1)  = \left(\begin{array}{c|c|c} 
			\begin{matrix}
				0 & 0 & 1 & -2 & 4 & -1   \\
				0 & 1 & 0 & 0 & -1 & 0 \\
				1 & 0 & 0 & 2 & -5 & 1\\
			    0 & 0 & 0 & 0 & -1 & 1 \\
			    0 & 0 & 0 & 0 & -1 & 0 \\
			    0 & 0 & 0 & 1 & -1 & 0 \\
			\end{matrix}
			& 
			\begin{matrix}
				0 & -2 & -2 & 0 & 0 & 0   \\
				0 & 0 & 0 & 0 & 0 & 0 \\
				0 & 2 & 2 & 0 & 0 & 0\\
			    0 & 0 & 0 & 0 & 0 & 0 \\
			    0 & 0 & 0 & 0 & 0 & 0 \\
			    0 & 0 & 0 & 0 & 0 & 0 \\
			\end{matrix}
			&
			\mbox{\Large $0$} 
			
			\\[0.5ex] & & \\[-5.5ex] & & \\[0.5ex] \hline & & \\[-2ex]
			\begin{matrix}
				0 & 0 & 0 & -\frac{3}{2} & 3 & -\frac{1}{2}   \\
				0 & 0 & 0 & \frac{1}{2} & -2 & -\frac{1}{2} \\
				0 & 0 & 0 & 0 & -1 & 0\\
			    0 & 0 & 0 & 0 & -1 & 0 \\
			    0 & 0 & 0 & -\frac{1}{2} & 1 & \frac{1}{2} \\
			    0 & 0 & 0 & \frac{3}{2} & -4 & \frac{1}{2} \\
			\end{matrix} 
			& 			
			\begin{matrix}
				0 & -2 & -1 & 0 & 0 & 1   \\
				0 & 1 & 0 & 0 & 0 & 0 \\
				0 & 0 & 1 & 0 & 0 & 0\\
			    0 & 1 & 0 & 0 & 1 & 0 \\
			    0 & -1 & 0 & 1 & 0 & 0 \\
			    1 & 2 & 1 & 0 & 0 & 0 \\
			\end{matrix}
			&\mbox{\Large $0$} 
			\\[0.5ex] & & \\[-5.5ex] & & \\[0.5ex] \hline & & \\[-2ex]
			
			\begin{matrix}
				0 & 0 & 0 & \frac{1}{2} & -1 & \frac{1}{2}   \\
				0 & 0 & 0 & -1 & 3 & 0 \\
				0 & 0 & 0 & -\frac{1}{2} & 2 & -\frac{1}{2}\\
			    0 & 0 & 0 & 1 & -2 & 0 \\
			    0 & 0 & 0 & \frac{1}{2} & 0 & \frac{1}{2} \\
			    0 & 0 & 0 & -\frac{1}{2} & 1 & -\frac{1}{2} \\
			\end{matrix}
			&
			\begin{matrix}
				0 & 1 & 1 & 0 & 0 & 0   \\
				0 & -1 & -1 & 0 & 0 & 0 \\
				0 & -1 & -1 & 0 & 0 & 0\\
			    0 & 1 & 1 & 0 & 0 & 0 \\
			    0 & 0 & 0 & 0 & 0 & 0 \\
			    0 & 0 & 0 & 0 & 0 & 0 \\
			\end{matrix}
			&
			\begin{matrix}
				0 & 1 & 0 & 0 & 0 & 0   \\
				1 & 0 & 0 & 0 & 0 & 0 \\
				0 & 0 & 0 & 1 & 0 & 0\\
			    0 & 0 & 1 & 0 & 0 & 0 \\
			    0 & 0 & 0 & 0 & 0 & 1 \\
			    0 & 0 & 0 & 0 & 1 & 0 \\
			\end{matrix}
		\end{array}\right).
	\]

A similar computation shows that the action of the rest of the generators is given by
	\[
		\rho(\pi_2)  = \left(\begin{array}{c|c|c} 
			\begin{matrix}
				0 & 1 & -1 & -1 & -1 & 4   \\
				1 & 0 & 1 & 0 & 0 & -1 \\
				0 & 0 & 1 & 1 & 1 & -5\\
			    0 & 0 & 0 & 1 & 0 & -1 \\
			    0 & 0 & 0 & 0 & 1 & -1 \\
			    0 & 0 & 0 & 0 & 0 & -1 \\
			\end{matrix}
			& 
			\begin{matrix}
				-2 & -2 & 0 & 0 & 0 & 0   \\
				0 & 0 & 0 & 0 & 0 & 0 \\
				2 & 2 & 0 & 0 & 0 & 0\\
			    0 & 0 & 0 & 0 & 0 & 0 \\
			    0 & 0 & 0 & 0 & 0 & 0 \\
			    0 & 0 & 0 & 0 & 0 & 0 \\
			\end{matrix}
			&
			\mbox{\Large $0$} 
			
    		\\[0.5ex] & & \\[-5.5ex] & & \\[0.5ex] \hline & & \\[-2ex]
			\begin{matrix}
				0 & 0 & -1 & -\frac{1}{2} & -\frac{3}{2} & 3   \\
				0 & 0 & 1 & -\frac{1}{2} & \frac{1}{2} & -2 \\
				0 & 0 & 0 & 0 & 0 & -1\\
			    0 & 0 & 0 & 0 & 0 & -1 \\
			    0 & 0 & -1 & \frac{1}{2} & -\frac{1}{2} & 1 \\
			    0 & 0 & 1 & \frac{1}{2} & \frac{3}{2} & -4 \\
			\end{matrix} 
			& 			
			\begin{matrix}
				-2 & -1 & 0 & 0 & 0 & 0   \\
				1 & 0 & 0 & 0 & 0 & 0 \\
				0 & 1 & 0 & 0 & 1 & 0\\
			    1 & 0 & 0 & 0 & 0 & 1 \\
			    -1 & 0 & 1 & 0 & 0 & 0 \\
			    2 & 1 & 0 & 1 & 0 & 0 \\
			\end{matrix}
			&\mbox{\Large $0$} 
			\\[0.5ex] & & \\[-5.5ex] & & \\[0.5ex] \hline & & \\[-2ex]
			
			\begin{matrix}
				0 & 0 & 0 & \frac{1}{2} & \frac{1}{2} & -1   \\
				0 & 0 & -1 & 0 & -1 & 3 \\
				0 & 0 & 0 & -\frac{1}{2} & -\frac{1}{2} & 2\\
			    0 & 0 & 1 & 0 & 1 & -2 \\
			    0 & 0 & 0 & \frac{1}{2} & \frac{1}{2} & 0 \\
			    0 & 0 & 0 & -\frac{1}{2} & -\frac{1}{2} & 1 \\
			\end{matrix}
			&
			\begin{matrix}
				1 & 1 & 0 & 0 & 0 & 0   \\
				-1 & -1 & 0 & 0 & 0 & 0 \\
				-1 & -1 & 0 & 0 & 0 & 0\\
			    1 & 1 & 0 & 0 & 0 & 0 \\
			    0 & 0 & 0 & 0 & 0 & 0 \\
			    0 & 0 & 0 & 0 & 0 & 0 \\
			\end{matrix}
			&
			\begin{matrix}
				1 & 0 & 0 & 0 & 0 & 0   \\
				0 & 0 & 0 & 1 & 0 & 0 \\
				0 & 0 & 1 & 0 & 0 & 0\\
			    0 & 1 & 0 & 0 & 0 & 0 \\
			    0 & 0 & 0 & 0 & 0 & 1 \\
			    0 & 0 & 0 & 0 & 1 & 0 \\
			\end{matrix}
		\end{array}\right).
	\]

and

	\[
		\rho(\pi_3)  = \left(\begin{array}{c|c|c} 
			\begin{matrix}
				0 & 4 & -2 & 1 & -1 & 0  \\
				0 & -1 & 0 & 0 & 1 & 0 \\
				0 & -5 & 2 & 0 & 1 & 0\\
			    0 & -1 & 0 & 0 & 0 & 1 \\
			    0 & -1 & 0 & 0 & 0 & 0 \\
			    1 & -1 & 1 & 0 & 0 & 0 \\
			\end{matrix}
			& 
			\begin{matrix}
				-2 & 0 & 0 & 0 & 0 & 0   \\
				0 & 0 & 0 & 0 & 0 & 0 \\
				2 & 0 & 0 & 0 & 0 & 0\\
			    0 & 0 & 0 & 0 & 0 & 0 \\
			    0 & 0 & 0 & 0 & 0 & 0 \\
			    0 & 0 & 0 & 0 & 0 & 0 \\
			\end{matrix}
			&
			\mbox{\Large $0$} 
			
			\\[0.5ex] & & \\[-5.5ex] & & \\[0.5ex] \hline & & \\[-2ex]
			\begin{matrix}
				0 & \frac{7}{2} & -\frac{3}{2} & 0 & -1 & 0   \\
				0 & -\frac{3}{2} & \frac{1}{2} & 0 & 1 & 0\\
				0 & -1 & 0 & 0 & 0 & 0\\
			    0 & -1 & 0 & 0 & 0 & 0 \\
			    0 & \frac{1}{2} & -\frac{1}{2} & 0 & -1 & 0 \\
			    0 & -\frac{9}{2} & \frac{3}{2} & 0 & 1 & 0 \\
			\end{matrix} 
			& 			
			\begin{matrix}
				-2 & 0 & 0 & 0 & 1 & 0   \\
				1 & 0 & 0 & 1 & 0 & 0 \\
				0 & 0 & 1 & 0 & 0 & 0\\
			    1 & 0 & 0 & 0 & 0 & 1 \\
			    -1 & 0 & 0 & 0 & 0 & 0 \\
			    2 & 1 & 0 & 0 & 0 & 0 \\
			\end{matrix}
			&\mbox{\Large $0$} 
			\\[0.5ex] & & \\[-5.5ex] & & \\[0.5ex] \hline & & \\[-2ex]
			\begin{matrix}
				0 & -\frac{3}{2} & \frac{1}{2} & 0 & 0 & 0   \\
				0 & 3 & -1 & 0 & -1 & 0 \\
				0 & \frac{5}{2} & -\frac{1}{2} & 0 & 0 & 0\\
			    0 & -2 & 1 & 0 & 1 & 0 \\
			    0 & -\frac{1}{2} & \frac{1}{2} & 0 & 0 & 0 \\
			    0 & \frac{3}{2} & -\frac{1}{2} & 0 & 0 & 0 \\
			\end{matrix}
			&
			\begin{matrix}
				1 & 0 & 0 & 0 & 0 & 0   \\
				-1 & 0 & 0 & 0 & 0 & 0 \\
				-1 & 0 & 0 & 0 & 0 & 0\\
			    1 & 0 & 0 & 0 & 0 & 0 \\
			    0 & 0 & 0 & 0 & 0 & 0 \\
			    0 & 0 & 0 & 0 & 0 & 0 \\
			\end{matrix}
			&
			\begin{matrix}
				0 & 0 & 0 & 1 & 0 & 0   \\
				0 & 0 & 0 & 0 & 1 & 0 \\
				0 & 1 & 0 & 0 & 0 & 0\\
			    0 & 0 & 0 & 0 & 0 & 1 \\
			    0 & 0 & 1 & 0 & 0 & 0\\
			    1 & 0 & 0 & 0 & 0 & 0 \\
			\end{matrix}
		\end{array}\right).
	\]
\subsection{Actions of the affine group on homology and monodromy of the 4-cover of the cube}
In this section, we compute the action of $\Aff(C)$ on the absolute homology of the translation cover of the cube. As a corollary of our computations, we obtain generators for the Kontsevich--Zorich monodromy group of $C$.

Recall, $\tilde \alpha\colon\Aff(C)\to \Sp(18,\mathbb R)$ denotes the representation arising from the action of the affine diffeomorphisms on $C$. In what follows, we actually compute the action of the Veech group and note that all the calculations and matrices only make sense up to the action of $\Aut(C)$. We let $\alpha\colon\Aff(C)\to \Sp(16,\mathbb R)$ denote the action on the zero-holonomy subspace.
The main result in this section is the following:

\begin{thm}
The Kontsevich--Zorich monodromy group of $C$ is generated by the following matrices

\[
    \alpha(g_1)\colonequals\left(
    \begin{array}{c|c|c}
    \begin{matrix} \mbox{\Large $\mathrm{Id}_5$}
    \end{matrix}
    & \begin{matrix} 0 & -1 & -7 & -8 & -1\\
    0 & 1 & 1 & 2 & 1\\
     0 & 0 & 8 & 8 & 0\\
     0 & 1 & 2 & 1 & 1\\
     0 & -1 & 0 & 1 & -1\\
    \end{matrix} & \begin{matrix}
     -\frac{1}{3} & -\frac{19}{3} & -\frac{13}{3} & -\frac{19}{3} & -\frac{16}{3} & -\frac{4}{3} \\
     \frac{4}{3} & \frac{7}{3} & \frac{4}{3} & \frac{7}{3} & \frac{1}{3} & \frac{1}{3} \\
     \frac{1}{3} & \frac{22}{3} & \frac{13}{3} & \frac{22}{3} & \frac{13}{3} & \frac{1}{3} \\
     \frac{1}{3} & \frac{7}{3} & \frac{1}{3} & \frac{7}{3} & \frac{4}{3} & \frac{4}{3} \\
      -1 & 1 & -1 & 1 & 0 & 0 \\
     \end{matrix}	\\[0.5ex] & & \\[-5.5ex] & & \\[0.5ex] \hline & & \\[-2ex]
     
    \begin{matrix}
    \mbox{\Large $0$} 
    \end{matrix} & \begin{matrix}-1 & -1 & -\frac{11}{2} & -\frac{15}{2} & -1\\
    0 & 0 & \frac{3}{2} & \frac{7}{2} & 1\\
     0 & 0 & 2 & 1 & 0\\
     0 & 0 & 1 & 2 & 0\\
     0 & 1 & -\frac{1}{2} & -\frac{5}{2} & 0\\
    \end{matrix} & \begin{matrix} -\frac{5}{3} & -\frac{31}{6} & -\frac{11}{3} & -\frac{43}{6} & -\frac{31}{6} & -\frac{7}{6} \\
     0 & \frac{3}{2} & 2 & \frac{7}{2} & \frac{1}{2} & \frac{1}{2} \\
    \frac{1}{3} & \frac{4}{3} & \frac{1}{3} & \frac{4}{3} & \frac{1}{3} & \frac{1}{3} \\
     \frac{1}{3} & \frac{4}{3} & \frac{1}{3} & \frac{4}{3} & \frac{1}{3} & \frac{1}{3} \\
     \frac{1}{3} & -\frac{1}{6} & -\frac{5}{3} & -\frac{13}{6} & -\frac{1}{6} & -\frac{1}{6} \\
    \end{matrix} \\[0.5ex] & & \\[-5.5ex] & & \\[0.5ex] \hline & & \\[-2ex]
    
    \begin{matrix}\mbox{\Large $0$}
    \end{matrix} & \begin{matrix}1 & 1 & \frac{5}{2} & \frac{5}{2} & 0 \\
    -1 & 0 & -4 & -6 & -1\\
    -1 & -1 & -\frac{7}{2} & -\frac{7}{2} & 0\\
    1 & 0 & 3 & 5 & 1\\
     1 & 1 & \frac{3}{2} & \frac{3}{2} & 0\\
     -1 & -1 & -\frac{5}{2} & -\frac{5}{2} & 0\\\end{matrix} & \begin{matrix} 2 & \frac{5}{2} & 1 & \frac{5}{2} & \frac{5}{2} & \frac{1}{2} \\
    -1 & -3 & -3 & -6 & -3 & -1 \\
    -\frac{4}{3} & -\frac{23}{6} & -\frac{1}{3} & -\frac{23}{6} & -\frac{17}{6} & -\frac{5}{6} \\
    \frac{2}{3} & \frac{8}{3} & \frac{8}{3} & \frac{17}{3} & \frac{8}{3} & \frac{2}{3} \\
    \frac{2}{3} & \frac{7}{6} & \frac{2}{3} & \frac{7}{6} & \frac{19}{6} & \frac{1}{6} \\
     -1 & -\frac{5}{2} & -1 & -\frac{5}{2} & -\frac{5}{2} & \frac{1}{2} \\
     \end{matrix}
    \end{array}
    \right),
\]

\[
    \begin{aligned}
        &\alpha(g_2) \colonequals \\
        &\left(
        \begin{array}{c|c|c}
            \begin{matrix}0 & -2 & -1 & -5 & 1\\
            -1 & 0 & -1 & 1 & 0\\
            0 & 1 & 1 & 6 & -1\\
            0 & 0 & 0 & 1 & 0\\
            1 & 1 & 1 & 2 & 1\\
            \end{matrix} &
            \begin{matrix}0 & 8 & 7 & 1 & 6\\
            0 & -2 & -1 & -1 & -2\\
            0 & -8 & -8 & 0 & -6\\
            0 & -1 & -2 & -1 & -1\\
            0 & -1 & 0 & 1 & -1\\
            \end{matrix} &
            \begin{matrix} \frac{13}{3} & \frac{1}{3} & \frac{13}{3} & \frac{1}{3} & \frac{4}{3} & \frac{4}{3} \\
            -\frac{7}{3} & -\frac{4}{3} & -\frac{7}{3} & -\frac{4}{3} & -\frac{1}{3} & -\frac{1}{3} \\
            -\frac{16}{3} & -\frac{1}{3} & -\frac{16}{3} & -\frac{1}{3} & -\frac{1}{3} & -\frac{1}{3} \\
            -\frac{7}{3} & -\frac{1}{3} & -\frac{7}{3} & -\frac{1}{3} & -\frac{4}{3} & -\frac{4}{3} \\
            1 & 3 & 1 & 3 & 2 & 2 \\
            \end{matrix} \\[0.5ex] & & \\[-5.5ex] & & \\[0.5ex] \hline & & \\[-2ex]
            
            \begin{matrix}0 & -2 & 0 & -\frac{9}{2} & 1\\
            -1 & 1 & -1 & \frac{3}{2} & 0\\
            0 & -1 & 0 & 1 & -1\\
            0 & 1 & 0 & 2 & 0\\
            0 & -1 & 0 & -\frac{3}{2} & 0\\ \end{matrix} &
            \begin{matrix}1 & \frac{15}{2} & \frac{11}{2} & 1 & \frac{11}{2}\\
            0 & -\frac{7}{2} & -\frac{3}{2} & 0 & -\frac{5}{2}\\
            0 & -1 & -2 & 0 & -1\\
            0 & -2 & -1 & 0 & -2\\
            0 & \frac{5}{2} & \frac{1}{2} & -1 & \frac{3}{2}\\ \end{matrix} & \begin{matrix}
            \frac{25}{6} & \frac{2}{3} & \frac{25}{6} & \frac{2}{3} & \frac{7}{6} & \frac{7}{6} \\
            -\frac{5}{2} & -1 & -\frac{5}{2} & -1 & -\frac{1}{2} & -\frac{1}{2} \\
            -\frac{7}{3} & -\frac{4}{3} & -\frac{7}{3} & -\frac{4}{3} & -\frac{4}{3} & -\frac{4}{3} \\
            -\frac{1}{3} & \frac{2}{3} & -\frac{1}{3} & \frac{2}{3} & \frac{2}{3} & \frac{2}{3} \\
            \frac{1}{6} & -\frac{1}{3} & \frac{1}{6} & -\frac{1}{3} & -\frac{5}{6} & -\frac{5}{6} \\
            \end{matrix} \\[0.5ex] & & \\[-5.5ex] & & \\[0.5ex] \hline & & \\[-2ex]
            
            \begin{matrix}1 & 0 & 0 & \frac{5}{2} & -1\\
            0 & -1 & 1 & -3 & 1\\
            -1 & 0 & 0 & -\frac{7}{2} & 1\\
            0 & 1 & -1 & 2 & -1\\
            1 & 1 & 1 & \frac{1}{2} & 0\\
            -1 & -1 & -1 & -\frac{3}{2} & 0\\ \end{matrix} & \begin{matrix}-1 & -\frac{5}{2} & -\frac{5}{2} & -1 & -\frac{5}{2}\\
            1 & 6 & 4 & 0 & 4\\
            1 & \frac{7}{2} & \frac{7}{2} & 1 & \frac{7}{2}\\
            -1 & -5 & -3 & 0 & -3\\
            -1 & -\frac{3}{2} & -\frac{3}{2} & -1 & -\frac{1}{2}\\
            1 & \frac{5}{2} & \frac{5}{2} & 1 & \frac{3}{2} \end{matrix} & \begin{matrix}
            -\frac{3}{2} & -1 & -\frac{3}{2} & 0 & -\frac{1}{2} & -\frac{1}{2} \\
            4 & 2 & 5 & 2 & 2 & 2 \\
            \frac{17}{6} & \frac{1}{3} & \frac{17}{6} & -\frac{2}{3} & \frac{5}{6} & \frac{5}{6} \\
            -\frac{11}{3} & -\frac{5}{3} & -\frac{14}{3} & -\frac{5}{3} & -\frac{5}{3} & -\frac{5}{3} \\
            \frac{5}{6} & \frac{4}{3} & \frac{5}{6} & \frac{4}{3} & \frac{5}{6} & -\frac{1}{6} \\
            \frac{1}{2} & -1 & \frac{1}{2} & -1 & -\frac{3}{2} & -\frac{1}{2} \\
            \end{matrix}
        \end{array}
        \right),
    \end{aligned}
\]
and
\[
    \arraycolsep=0.9951\arraycolsep
    \begin{aligned}
        &\alpha(g_3)\colonequals \\
        &\left(
        \begin{array}{c|c|c}\begin{matrix}-1 & 5 & -1 & 1 & 1\\
        0 & -1 & 0 & 1 & 0\\
        1 & -6 & 1 & -1 & 0\\
        -1 & -2 & -1 & -1 & -1\\
        0 & -1 & 0 & 0 & 0\\ \end{matrix} & \begin{matrix}
        0 & -4 & 1 & -7 & -6\\
        0 & 0 & -1 & 1 & 0\\
        0 & 6 & 0 & 8 & 8\\
        0 & 1 & -1 & 0 & 1\\
        0 & 1 & 1 & 2 & 1 \end{matrix} & \begin{matrix}
        0 & 1 & 0 & 1 & -3 & -3 \\
        -1 & 0 & -1 & 0 & 1 & 1 \\
        \frac{1}{3} & \frac{1}{3} & \frac{1}{3} & \frac{1}{3} & \frac{16}{3} & \frac{16}{3} \\
        -2 & -3 & -2 & -3 & -1 & -1 \\
        \frac{4}{3} & \frac{1}{3} & \frac{4}{3} & \frac{1}{3} & \frac{7}{3} & \frac{7}{3} \\
        \end{matrix} \\[0.5ex] & & \\[-5.5ex] & & \\[0.5ex] \hline & & \\[-2ex]
        
        \begin{matrix}-1 & \frac{9}{2} & -\frac{1}{2} & 0 & 1\\
        0 & -\frac{3}{2} & \frac{1}{2} & 1 & 0\\
        0 & -2 & 0 & -1 & 0\\
        1 & -1 & 1 & 1 & 0\\
        -1 & \frac{1}{2} & -\frac{3}{2} & -1 & -1\\ \end{matrix} & \begin{matrix}0 & -\frac{7}{2} & 0 & -\frac{13}{2} & -\frac{9}{2}\\
        1 & \frac{3}{2} & 0 & \frac{7}{2} & \frac{3}{2}\\
        0 & 2 & 0 & 1 & 2\\
        0 & 1 & 0 & 2 & 1\\
        -1 & -\frac{3}{2} & -1 & -\frac{7}{2} & -\frac{3}{2}\\ \end{matrix} & \begin{matrix}
        \frac{1}{6} & \frac{2}{3} & \frac{1}{6} & \frac{2}{3} & -\frac{17}{6} & -\frac{17}{6} \\
        \frac{1}{2} & 1 & \frac{1}{2} & 1 & \frac{5}{2} & \frac{5}{2} \\
        -\frac{2}{3} & -\frac{2}{3} & -\frac{2}{3} & -\frac{2}{3} & \frac{1}{3} & \frac{1}{3} \\
        \frac{4}{3} & \frac{4}{3} & \frac{4}{3} & \frac{4}{3} & \frac{7}{3} & \frac{7}{3} \\
        -\frac{5}{2} & -3 & -\frac{5}{2} & -3 & -\frac{7}{2} & -\frac{7}{2} \\ \end{matrix} \\[0.5ex] & & \\[-5.5ex] & & \\[0.5ex] \hline & & \\[-2ex]
        
        \begin{matrix}0 & -\frac{5}{2} & -\frac{1}{2} & -1 & -1\\
        0 & 4 & 0 & 0 & 0\\
        -1 & \frac{5}{2} & -\frac{1}{2} & 0 & 0\\
        1 & -2 & 1 & 0 & 0\\
        0 & -\frac{1}{2} & -\frac{1}{2} & 0 & 0\\
        0 & \frac{3}{2} & \frac{1}{2} & 1 & 1\\ \end{matrix} & \begin{matrix}0 & \frac{3}{2} & 0 & \frac{5}{2} & \frac{5}{2}\\
        0 & -3 & 0 & -5 & -4\\
        0 & -\frac{5}{2} & 0 & -\frac{7}{2} & -\frac{7}{2}\\
        0 & 2 & 0 & 4 & 3\\
        0 & \frac{1}{2} & 0 & \frac{1}{2} & \frac{1}{2}\\
        0 & -\frac{3}{2} & 0 & -\frac{3}{2} & -\frac{3}{2}\\ \end{matrix} & \begin{matrix} -\frac{1}{2} & -2 & -\frac{1}{2} & -1 & \frac{1}{2} & \frac{1}{2} \\
        \frac{1}{3} & \frac{1}{3} & \frac{1}{3} & \frac{1}{3} & -\frac{8}{3} & -\frac{11}{3} \\
        -\frac{1}{2} & 0 & -\frac{1}{2} & -1 & -\frac{5}{2} & -\frac{5}{2} \\
        \frac{2}{3} & \frac{2}{3} & \frac{2}{3} & \frac{2}{3} & \frac{5}{3} & \frac{8}{3} \\
        \frac{1}{6} & -\frac{1}{3} & -\frac{5}{6} & -\frac{1}{3} & \frac{1}{6} & \frac{1}{6} \\
        -\frac{1}{6} & \frac{4}{3} & \frac{5}{6} & \frac{4}{3} & -\frac{1}{6} & -\frac{1}{6} \\ \end{matrix} \end{array}
        \right),
    \end{aligned}
\]

where $g_1 = \begin{pmatrix} 1 & 2 \\ 0 & 1\end{pmatrix},$ $g_2 = \begin{pmatrix} 5 & -2 \\ 3 & -1\end{pmatrix},$ and $g_3 = \begin{pmatrix} 3 & -2 \\ 5 & -3\end{pmatrix}.$
\end{thm}

\begin{proof} We apply the same computation as in the previous example. Here we show figures that describe the action of $g_1$ on $C$ (Figure~\ref{fig:g1(C2)}) and $g_3$ on $C$ (Figure~\ref{fig:g3(C2)}). Since $g_1$ is a horizontal shear map, the horizontal curves are preserved, and all other curves' slopes decrease. However, $g_3$ increases the slopes of all curves, hence a vertical arrangement of the squares is easier to follow (opposite horizontal sides are identified). We leave $g_2$ as an exercise for the reader. 

\begin{figure}[htbp] 
	\centering
    \includegraphics[width=6in]{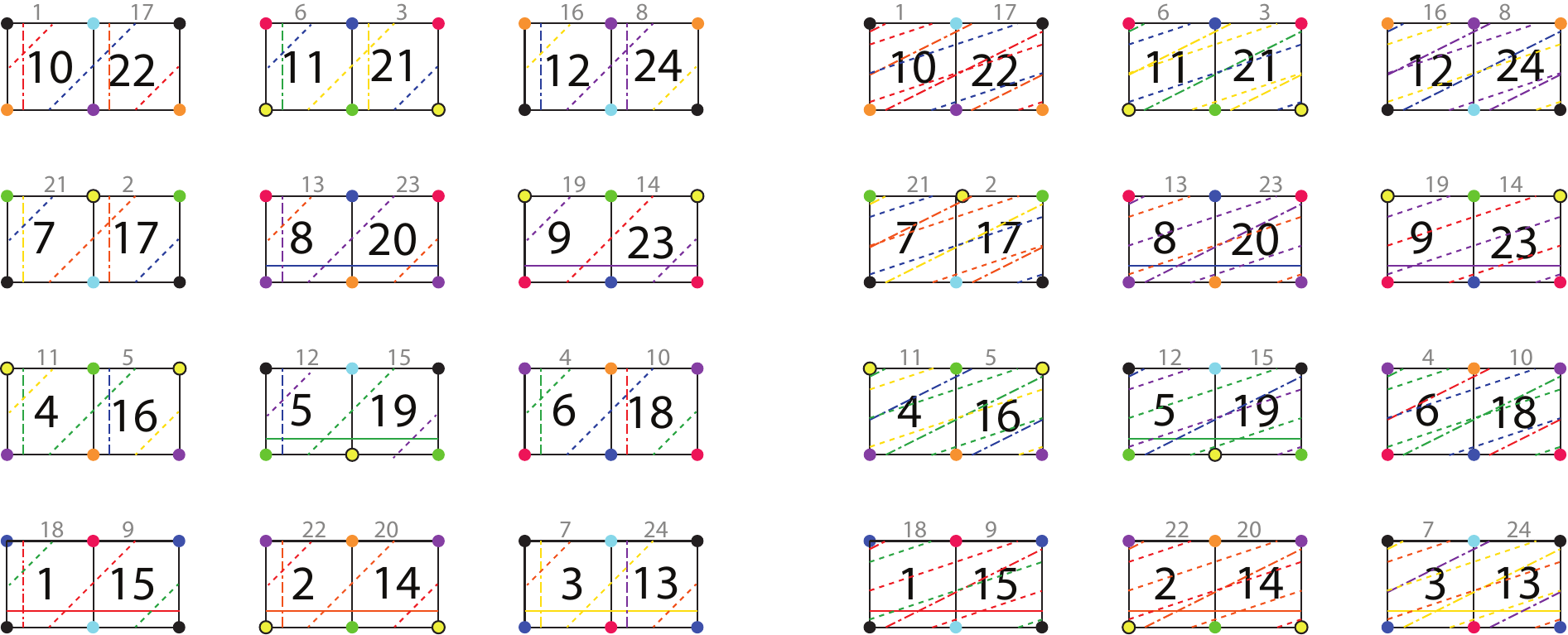}
    \caption{$C$ (left) and $g_1(C)$ (right)}
    \label{fig:g1(C2)}
\end{figure}

\begin{figure}[htbp] 
	\centering
    \includegraphics[width=5.5in]{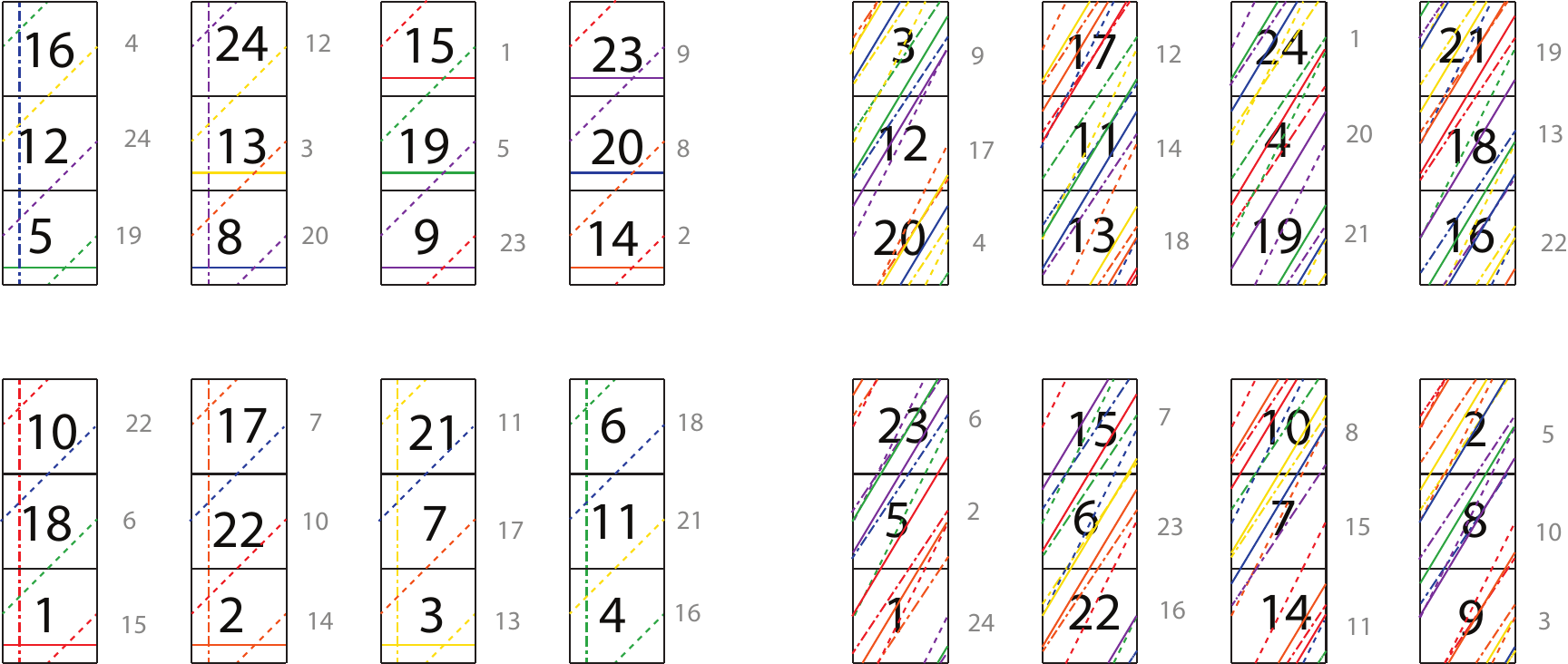}
    \caption{$C$ (left) and $g_3(C)$ (right)}
    \label{fig:g3(C2)}
\end{figure}

\end{proof}

\subsection{Identification of the monodromy group}
In this section, we use the representations arising from the automorphisms $\rho$ and affine diffeomorphisms $\alpha$ to identify the monodromy group.

\begin{thm} The Zariski closure of the monodromy group of the translation cover of the cube is $$\overline{\alpha(\Aff(C))}^{\Zar}\simeq \Sp(2,\R)^3\simeq \SL(2,\R)^3.$$ \end{thm}

\begin{proof}

We begin by computing the isotypical components of the homology group $H_1(C;\R)$. Recall that $\Aut(C)\simeq S_4$ and note that all representations are either 1, 2 or 3 dimensional (see e.g.\ Fulton and Harris \cite{MR1153249}).
 
\begin{enumerate}
 \item 1-dimensional representations
 \end{enumerate}
 
These are simultaneous eigenvectors for the action of $\Aut(C).$ The eigenvalues of $\rho(\pi_1)$ and $\rho(\pi_2)$ are either $\pm1,$ and the eigenvalues of $\rho(\pi_3)$ are the sixth roots of unity. The common eigenvectors arise as vectors corresponding to eigenvalue 1. Let $$V(\lambda_1,\lambda_2,\lambda_3) = \{v\in H_1(C;\R):\rho(\pi_i)v=\lambda_i v, i=1,2,3\}$$
denote the simultaneous eigenspace corresponding to the ordered eigenvalues $\lambda_1,\lambda_2,\lambda_3$. Consider \begin{align*}
E_1 &\colonequals\text{span}_\R \{2\sigma_1-2\sigma_3 +\zeta_1-\zeta_2-\zeta_3-\zeta_4-2\zeta_8+2\eta_2+2\eta_3+\eta_6+\eta_8\},\\
E_2 &\colonequals \text{span}_\R \{-\sum \eta_i\}.
\end{align*}

Then,
$ V(1,1,1)=E_1\oplus E_2.$

\begin{enumerate}[resume]
\item 2-dimensional irreducible representations
\end{enumerate}

We pick a vector that is an eigenvector for both $\rho(\pi_1)$ and $\rho(\pi_2),$ but not $\rho(\pi_3):$  
$$v_1=8\sigma_1 -2\sigma_2 -10\sigma_3 -4\sigma_5 -4\sigma_8+ 7\zeta_1 -3\zeta_2 -2\zeta_3 -2\zeta_4 +\zeta_5 -9\zeta_8 +3\eta_2+ 8\eta_3 -7\eta_4 -\eta_6+ 3\eta_8$$ satisfies $\rho(\pi_1) v_1 = - v_1$ and $\rho(\pi_2) v_1 = v_1.$ Under $\rho(\pi_3),$ it yields $$\rho(\pi_3)v_1 = -2\sigma_1 -2\sigma_2+ 2\sigma_5 +2\sigma_8 -\zeta_1 -\zeta_2+ \zeta_5+ \zeta_8 -2\eta_1 -4\eta_3+ 4\eta_6+ 2\eta_8.$$ One can check that the rank of $\{v_1, \rho(\pi_3)v_1, \rho(\pi_1)\rho(\pi_3)v_1, \rho(\pi_2)\rho(\pi_3)v_1\}$ is two. We call $L_1 = \text{span}_\R \{v_1, \rho(\pi_3)v_1\}.$

Similarly, $v_2=\eta_1-\eta_2+\eta_3-\eta_4$ also satisfies $\rho(\pi_1) v_2 = - v_2$ and $\rho(\pi_2) v_2 = v_2,$ and $L_2 = \text{span}_\R \{\eta_1-\eta_2+\eta_3-\eta_4,-\eta_1-\eta_3+\eta_6+\eta_8\}$ is preserved under $\Aut(C).$ The latter is obtained by $\rho(\pi_3) v_2.$

We show that the action of $\Aut(C)$ restricted to $L_1$ and $L_2$ are conjugate to each other. On $L_1$ with respect to the basis $\{v_1, \rho(\pi_3) v_1\},$ we have  $$\rho(\pi_1)|_{L_1} = \begin{pmatrix} -1 & 1 \\ 0 & 1 \end{pmatrix}, \qquad \rho(\pi_2)|_{L_1} = \begin{pmatrix} 1 & 0 \\ 0 & 1\end{pmatrix}, \qquad \rho(\pi_3)|_{L_1} = \begin{pmatrix} 0 & -1 \\ 1 & -1\end{pmatrix}.$$ In fact, we get $\rho(\pi_i)|_{L_1} = \rho(\pi_i)|_{L_2},$ hence the two representations are equivalent. Furthermore, the centralizer of $\rho(\Aut(C))|_{L_i}$ is $\{a \mathrm{Id}_2:a\in \R\}\simeq \R,$ hence, the representation is real.

\begin{enumerate}[resume] \item 3-dimensional irreducible representations
\end{enumerate}

Finally, we decompose the rest into four three-dimensional representations, and show that they are two pairs of equivalent representations but not all four are equivalent. 

Let $$H_1 = \text{span}_\R \{\eta_1-\eta_2-\eta_3+\eta_4-\eta_6+\eta_8, \eta_1+\eta_2-\eta_3-\eta_4+\eta_6-\eta_8,-\eta_1+\eta_2+\eta_3-\eta_4-\eta_6+\eta_8\},$$ \begin{align*}
H_2 = \text{span}_\R\{&\sigma_1 -\sigma_3, \sigma_1 -\sigma_3 +\zeta_1 -\zeta_2 +\zeta_5 -\zeta_8 +\eta_2- \eta_4,\\
&4\sigma_1 -4\sigma_3+ 3\zeta_1 -\zeta_2+ \zeta_5 -3\zeta_8 -\eta_1+ 2\eta_2+ \eta_3 -2\eta_4 -\eta_6+\eta_8\}.
\end{align*}
Then
\begin{align*}
    \rho(\pi_1)|_{H_1} &= \begin{pmatrix} -1 & 1 & 0\\ 0 & 1 & 0\\
    0 & 1 & -1\end{pmatrix}, & \rho(\pi_2)|_{H_1} &= \begin{pmatrix} 0 & 1 & -1 \\
    1 & 0 & -1\\ 
    0 & 0 & -1\end{pmatrix}, & \rho(\pi_3)|_{H_1} &= \begin{pmatrix} 1 & 0 & -1\\
    0 & 0 & -1\\
    0 & 1 & -1\end{pmatrix},
\intertext{and}
    \rho(\pi_1)|_{H_2} &= \begin{pmatrix} -1 & 0 & 0\\
    0 & 1 & 0\\
    0 & 0 & -1\end{pmatrix}, & \rho(\pi_2)|_{H_2} &= \begin{pmatrix} 0 & 1 & 0\\
    1 & 0 & 0\\
    0 & 0 & -1\end{pmatrix}, & \rho(\pi_3)|_{H_2} &= \begin{pmatrix} \frac{1}{2} & -\frac{1}{2} & 2 \\
    \frac{1}{2} & -\frac{1}{2} & -2\\
    \frac{1}{4} & \frac{1}{4} & 0\end{pmatrix}.
\end{align*}

Again, the centralizer of $\rho(\Aut(C))|_{H_1}$ is $\{a \mathrm{Id}_3:a\in \R\}\simeq \R,$ hence the representations are real and we have 
$$\rho(\pi_i)|_{H_1} \begin{pmatrix} 1 & 0 & -1\\ 1 & 2 & 1\\ 2 & 0 & 2\end{pmatrix}
= \begin{pmatrix} 1 & 0 & -1\\ 1 & 2 & 1\\ 2 & 0 & 2\end{pmatrix}
\rho(\pi_i)|_{H_2},$$ for all $i=1,2,3,$ hence the two representations $\rho(\Aut(C_1))|_{H_i}$ are equivalent up to conjugation.

Let \begin{align*}
P_1 = \text{span}_\R \{&\sigma_1 -\sigma_2 -\sigma_5+ \sigma_8, \, -\sigma_1+ \sigma_2 -\sigma_5+ \sigma_8 -\zeta_1+ \zeta_2 -\zeta_5+ \zeta_8 -\eta_2+ \eta_4,\\
&12\sigma_1 -4\sigma_2 -12\sigma_3 -2\sigma_5 -2\sigma_8 -4\sigma_9+ 9\zeta_1 -5\zeta_2 -2\zeta_3 
-2\zeta_4+ 3\zeta_5 -11\zeta_8 {}-{} \\
&3\eta_1+8\eta_2+ 5\eta_3 -6\eta_4 -\eta_6+ 3\eta_8\},
\end{align*}
\begin{align*}
    P_2 = \text{span}_\R \{&{-\zeta_1}+ \zeta_2 -\zeta_3+ \zeta_4, \, 2\sigma_1 -2\sigma_3+ 2\zeta_1-\zeta_3 -\zeta_4+ \zeta_5 -\zeta_8 -\eta_1+ \eta_2+ \eta_3 -\eta_4,\\
    &{-\zeta_3}+ \zeta_4 -\zeta_5+ \zeta_8\}.
\end{align*}
Then we have,
\begin{align*}
    \rho(\pi_1)|_{P_1} &= \begin{pmatrix} 0 & -1 & 2\\ 0 & 1 & 0\\
    \frac{1}{2} & \frac{1}{2} & 0\end{pmatrix}, & \rho(\pi_2)|_{P_1} &= \begin{pmatrix} 0 & 1 & 0 \\
    1 & 0 & 0\\ 
    0 & 0 & -1\end{pmatrix}, & \rho(\pi_3)|_{P_1} &= \begin{pmatrix} 0 & -1 & 2\\
    0 & 1 & 0\\
    -\frac{1}{2} & \frac{1}{2} & -1\end{pmatrix},
\intertext{and}
    \rho(\pi_1)|_{P_2} &= \begin{pmatrix} 1 & 0 & 0\\ 0 & 0 & 1\\
    0 & 1 & 0\end{pmatrix}, & \rho(\pi_2)|_{P_2} &= \begin{pmatrix} -1 & 0 & 0 \\
    0 & -1 & 0\\ 
    0 & 0 & 1\end{pmatrix}, & \rho(\pi_3)|_{P_2} &= \begin{pmatrix} 0 & 0 & 1\\
    1 & 0 & 0\\
    0 & 1 & 0\end{pmatrix}.
\end{align*}

With the same argument as above, the representations are real, and the two representations are equivalent. However, 

$$\rho(\pi_i)|_{P_1} \begin{pmatrix} 1 & 1 & 2\\ 0 & 2 & 0\\ -1 & 1 & 2\end{pmatrix}
= \begin{pmatrix} 1 & 1 & 2\\ 0 & 2 & 0\\ -1 & 1 & 2\end{pmatrix}
\rho(\pi_i)|_{P_2},$$ but neither $\rho(\Aut(C))|_{P_1}$ nor $\rho(\Aut(C))|_{P_2}$ is equivalent to $\rho(\Aut(C))|_{H_i}.$

Thus, we have decomposed $H_1(C;\R)$ into its irreducible pieces:
$$H_1(C;\R)= (E_1\oplus E_2)\oplus(L_1\oplus L_2)\oplus (H_1 \oplus H_2 )\oplus(P_1\oplus P_2).$$
The isotypical components are those inside the same pair of parentheses. Notice that, for dimensional considerations, $E_1\oplus E_2$ corresponds to the tautological plane.
    
By the results of Matheus--Yoccoz--Zmiaikou \cite{MYZ}, we have the Zariski closure of the monodromy group
\begin{itemize}
    \item  $\tilde \alpha(\Aff(C))|_{L_1\oplus L_2}$ is contained in $\Sp(2,\R)$,
     \item  $\tilde \alpha(\Aff(C))|_{H_1 \oplus H_2}$ is contained in $\Sp(2,\R)$, and
    \item  $\tilde \alpha(\Aff(C))|_{P_1\oplus P_2}$ is contained in $\Sp(2,\R)$.
\end{itemize}

Thus, the upper bound for the dimension of the Zariski closure of the full monodromy group is 3 + 3 + 3 = 9.

On the other hand, we find a list of 9 elements in the Lie algebra of $\overline{\alpha(\Aff(C))}^{\Zar}.$

We denote $A=\alpha(g_1), B=\alpha(g_2),$ and $C=\alpha(g_3).$ Since $A$ is parabolic, we take $a=\text{log}(A)$ and denote by $\phi_X$ the conjugation map $\phi_X(a) = X a X^{-1}.$ Then $$a, \, \phi_B(a), \, \phi_{B^2}(a), \, \phi_{B^3}(a), \, \phi_C(a), \, \phi_{BC}(a), \, \phi_{B^2C}(a), \, \phi_{AC}(a), \, \phi_{A^2C}(a)$$ form a linearly independent set inside the Lie algebra of $\overline{\alpha(\Aff(C))}^{\Zar}$.
\end{proof}
\subsection{Lyapunov exponents of the 4-cover of the cube}

In this section we compute the Lyapunov spectrum of the 4-cover of the cube $C$.

\begin{prop}
Counting multiplicities, the positive Lyapunov spectrum of the translation cover of the cube is
$$\{1\} \cup \{2/3,2/3,2/3\} \cup \{1/3,1/3,1/3\} \cup \{1/3,1/3\},$$
where the unions indicate the exponents corresponding to distinct, symplectically-orthogonal, irreducible pieces of the Hodge bundle.
\end{prop}

\begin{proof}
Recall the following decomposition of the homology of $C$ into isotypical components,
$$H_1(C;\R)= (E_1\oplus E_2)\oplus(L_1\oplus L_2)\oplus (H_1 \oplus H_2 )\oplus(P_1\oplus P_2).$$
The subspace $E_1\oplus E_2$ is the tautological plane and carries a Lyapunov exponent of 1.

Let $C_{i,j} = C/\langle \pi_i,\pi_j\rangle$ for $i,j\in\{1,2,3\}$. In case $i=j$ we denote $C_{i,j}$ simply by $C_i$. We note that $C_{1,2}$ and $C_{1,3}$ are genus 2 surfaces. Thus, by the Eskin--Kontsevich--Zorich formula \cite{MR3270590} and the \texttt{surface\_dynamics} package, we obtain that the sum of the Lyapunov exponents on $C_{1,2}$ and $C_{2,3}$ are $ 4/3$. That is, we obtain a Lyapunov exponent of $1/3$ on $H_1(C_{1,2};\mathbb R)$ and $H_1(C_{2,3};\mathbb R)$. 

We now try to understand how the homology group of the quotients sit inside the homology group of $C$. In particular, we find in which isotypical components the homology of the quotients lie in. Once we have one Lyapunov exponent in an isotypical component, we use Proposition \ref{prop_MYZ_exp} to argue that we have all of them in the isotypical component.

Since the deck group of the covering $C \to C_{1,2}$ is exactly the $\langle \pi_1, \pi_2\rangle$, we can identify
\[
    H_1(C_{1,2};\mathbb R)=\{v\in H_1(C;\mathbb R): \pi_1v=v,\pi_2 v =v\}.
\]

Notice that the intersection of $H_1(C_{1,2}; \mathbb R)$ and $L_1\oplus L_2$ is non-empty. (It is spanned by $\eta_1+\eta_2+\eta_3+\eta_4-2\eta_6-2\eta_8$ and $-4\sigma_1+6\sigma_2+10\sigma_3-5\zeta_1+5\zeta_2+2\zeta_3+\zeta_4-3\zeta_5+7\zeta_8+4\eta_1-3\eta_2+7\eta_4-7\eta_6-7\eta_8$.)

By Proposition \ref{prop_MYZ_exp}, the multiplicity of 1/3 in $L_1\oplus L_2$ is 2.

A similar computation shows that the intersection of $H_1(C_{1,3}; \mathbb R)$ and $P_1\oplus P_2$ is nonempty. (It is spanned by
$-2\sigma_1 + 2\sigma_3 -\zeta_1 -\zeta_2 + 3\zeta_3 -\zeta_4 + \eta_1 -\eta_2 -\eta_3 + \eta_4$ and 
 $-8\sigma_1 + 12\sigma_3 + 4\sigma_5 + 4\sigma_9 -6\zeta_1 + 2\zeta_2 + 2\zeta_3 + 2\zeta_4 + 8\zeta_8 + 3\eta_1 -5\eta_2 -5\eta_3 + 3\eta_4 + \eta_6 -3\eta_8$.)

By Proposition \ref{prop_MYZ_exp}, the multiplicity of 1/3 in $P_1\oplus P_2$ is 3.

Lastly, the Eskin--Kontsevich--Zorich formula \cite{MR3270590} and the \texttt{surface\_dynamics} package yields that the sum of the Lyapunov exponents on $C$ must be 14/3. On the other hand, by Proposition \ref{prop_MYZ_exp}, the multiplicity of the Lyapunov exponent in the remaining isotypical component $H_1\oplus H_2$ is 3. Combining these results along with the Lyapunov exponents we previously obtained proves that the isotypical component $H_1\oplus H_2$ carries a Lyapunov exponent of 2/3 with multiplicity 3.
\end{proof}

\section{On the 2-cover of the mutetrahedron}\label{sec: mut}

\subsection{Unfolding of the mutetrahedron as a square-tiled surface}
Mutetrahedron is infinite polyhedral surface tiled by hexagons that is invariant under a rank-three lattice in $\R^3.$ Its quotient under the lattice is a genus three Riemann surface tiled by four hexagons. The mutetrahedron is a half-translation surface and its translation cover of genus five is studied in \cite{AL}. We note that the Veech group of the translation cover of the mutetrahedron is the same as the Veech group of the translation cover of the octahedron.

We will denote  by $M$ the image under $T^2$ of the translation cover from \cite{AL}.

The generators of the Veech group are
$$\SL(M)=\langle S^{-1}T, TST^{-1}\rangle < \SL(2,\Z).$$


Figure~\ref{fig:pi1(MT4)} describes $M$ with its homology basis.

\begin{figure}[htbp] 
	\centering
    \includegraphics[width=6in]{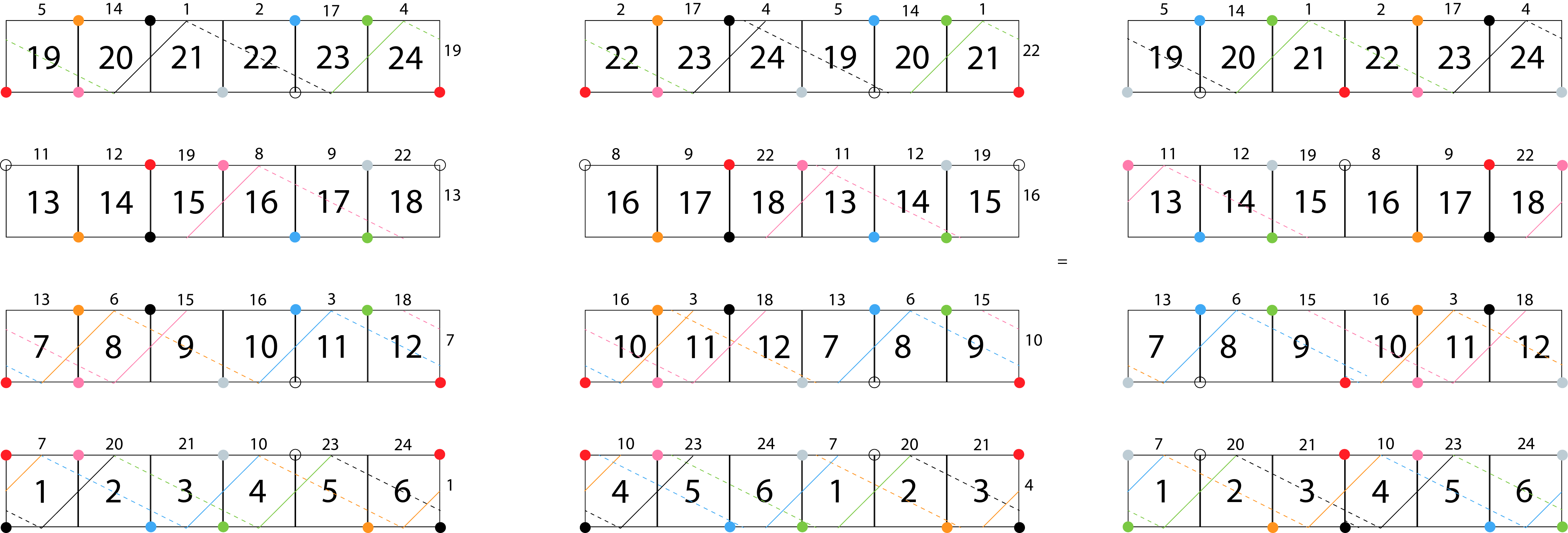}
    \caption{$M$ (left), action of $\pi_1$ on $M$ (center), $\pi_1(M)$ after cut-and-paste (right)}
    \label{fig:pi1(MT4)}
\end{figure}

\subsection{Basis of homology and zero-holonomy}\label{subsec: mut intersection} 

The absolute homology of the translation cover $M$ is 10-dimensional and an explicit basis  can be given by the curves $\sigma_i$ that start on the bottom of square $i$ with holonomy $\left(\begin{matrix}
2   \\
2 \\
\end{matrix}\right)$ and $\zeta_i$ that start on the bottom of square $i$ with holonomy $\left(\begin{matrix}
-4   \\
2 \\
\end{matrix}\right)$ where $i=1,3,4,6,8$ (See Figure~\ref{fig:pi1(MT4)}). 

We define the basis on the zero-holonomy subspace by $\Sigma_i \colonequals \sigma_i - \sigma_{8}$ and $Z_i  \colonequals \zeta_i - \zeta_8 $ for $i=1,3,4,6,8$. 

\subsection{Intersection form}
We record the intersection matrix encoding the algebraic intersection form of $M$ with respect to the basis given above. 

	\[
		\Omega= \left(\begin{array}{c|c} 
			 \mbox{\Large $0$} & \begin{matrix}
				-1 & 0  & -1 & -1 & 0  \\
				-1 & -1 &  0 & -1 & 0 \\
				-1 & -1 & -1 &  0 & 0 \\
				0  & -1 & -1 & -1 & -1  \\
				0  &  0 &  0 & -1 & -1  \\
			\end{matrix}  \\[0.5ex] & \\[-5.5ex] & \\[0.5ex] \hline & \\[-2ex]
			 \begin{matrix}
				1 & 1 & 1 & 0 & 0  \\
				0 & 1 & 1 & 1 & 0 \\
				1 & 0 & 1 & 1 & 0 \\
				1 & 1 & 0 & 1 & 1  \\
				0 & 0 & 0 & 1 & 1  \\
			\end{matrix} & \mbox{\Large $0$} 
		\end{array}\right).
	\]

\subsection{Action of the automorphism group on homology}\label{sec: mt aut}
Using \texttt{surface\_dynamics}, we get $\Aut(M)\simeq (\Z/2\Z)^3,$ the elementary abelian group of order eight and that it is generated by the permutations 
\begin{align*}
\pi_1 & = (1,4)(2,5)(3,6)(7,10)(8,11)(9,12)(13,16)(14,17)(15,18)(19,22)(20,23)(21,24)\\
\pi_2 & = (1,8)(2,9)(3,10)(4,11)(5,12)(6,7)(13,24)(14,19)(15,20)(16,21)(17,22)(18,23) \\
\pi_3 & =  (1,18)(2,13)(3,14)(4,15)(5,16)(6,17)(7,22)(8,23)(9,24)(10,19)(11,20)(12,21).
\end{align*}

Following a similar computation from the previous sections, we let $\rho$ denote the representation arising from the action of $\Aut(M)$ on $H_1(M;\R)$. Then the action of $\pi_1$ on $H_1(M;\R)$ (see Figure~\ref{fig:pi1(MT4)}) is given by
	\[
		\rho(\pi_1)  = \left(\begin{array}{c|c} 
			\begin{matrix}
				0 & 0 & 1 & 0 & 1   \\
				0 & 0 & 0 & 1 & 0 \\
				1 & 0 & 0 & 0 & -1\\
			    0 & 1 & 0 & 0 & 0 \\
			    0 & 0 & 0 & 0 & 1 \\
			\end{matrix}  & \mbox{\Large $0$} \\[0.5ex] & \\[-5.5ex] & \\[0.5ex] \hline & \\[-2ex]
			 \mbox{\Large $0$}  & 
			 \begin{matrix}
				0 & 0 & 1 & 0 & -1   \\
				0 & 0 & 0 & 1 & 0 \\
				1 & 0 & 0 & 0 & 1\\
			    0 & 1 & 0 & 0 & 0 \\
			    0 & 0 & 0 & 0 & 1 \\
			\end{matrix}
		\end{array}\right).
	\]
A similar computation shows the action of $\pi_2$ and $\pi_3$ on $H_1(M;\R)$ is given by
	\[
		\rho(\pi_2)  = \left(\begin{array}{c|c} 
			\begin{matrix}
				0 & 0 & 1 & 0 & 1   \\
				0 & 1 & 0 & 0 & 0 \\
				0 & 0 & -1 & 0 & 0\\
			    0 & 0 & 0 & 1 & 0 \\
			    1 & 0 & 1 & 0 & 0 \\
			\end{matrix}  & \mbox{\Large $0$} \\[0.5ex] & \\[-5.5ex] & \\[0.5ex] \hline & \\[-2ex]
			 \mbox{\Large $0$}  & 
			 \begin{matrix}
				0 & 0 & -1 & 0 & 1   \\
				0 & 0 & 0 & 1 & 0 \\
				0 & 0 & 1 & 0 & 0\\
			    0 & 1 & 0 & 0 & 0 \\
			    1 & 0 & 1 & 0 & 0 \\
			\end{matrix}
		\end{array}\right),
	\]
and
	\[
		\rho(\pi_3)  = \left(\begin{array}{c|c} 
			\begin{matrix}
				1 & 1 & 0 & 1 & 0   \\
				0 & 0 & 0 & -1 & 0 \\
				-1 & 0 & 0 & 0 & 1\\
			    0 & -1 & 0 & 0 & 0 \\
			    1 & 1 & 1 & 1 & 0 \\
			\end{matrix}  & \mbox{\Large $0$} \\[0.5ex] & \\[-5.5ex] & \\[0.5ex] \hline & \\[-2ex]
			 \mbox{\Large $0$}  & 
			 \begin{matrix}
				1 & 1 & 0 & 1 & 0   \\
				0 & 0 & 0 & -1 & 0 \\
				-1 & 0 & 0 & 0 & 1\\
			    0 & -1 & 0 & 0 & 0 \\
			    1 & 1 & 1 & 1 & 0 \\
			\end{matrix}
		\end{array}\right).
	\]

\subsection{Action of the affine group on homology and monodromy of the 2-cover of the mutetrahedron}\label{subsec: mut aff}
In this section, we compute the action of $\Aff(M)$ on the absolute homology of the translation cover of the mutetrahedron. As a corollary of our computations, we obtain generators for the Kontsevich--Zorich monodromy group of $M$.

Recall, $\tilde \alpha\colon\Aff(M)\to \Sp(10,\mathbb R)$ denotes the representation arising from the action of the affine diffeomorphisms on $M$. In what follows, we actually compute the action of the Veech group and note that all the calculations and matrices only make sense up to the action of $\Aut(M)$. We let $\alpha\colon\Aff(M)\to \Sp(8,\mathbb R)$ denote the action on the zero-holonomy subspace.
The main result in this section is the following:

\begin{thm}
The Kontsevich--Zorich monodromy group of $M$ is generated by the following two matrices \[
\alpha(S^{-1}T) = \left(\begin{array}{c|c} 
		\mbox{\Large $0$} & 
			\begin{matrix} 0 & 0 & -1 & 0\\
                    -1 & -1 & -1 & -1\\
                    0 & 0 & 1 & 1\\
                    1 & 0 & 1 & 0\end{matrix} \\[0.5ex] & \\[-5.5ex] & \\[0.5ex] \hline & \\[-2ex]
			 \begin{matrix}
                    0 & 0 & -1 & 0 \\
                    0 & 1 & 0 & 0 \\
                    1 & 0 & 1 & 0 \\
                    0 & 0 & 0 & 1			 \end{matrix}  & 
			 \begin{matrix}
			   1 & 0 & 1 & 1\\
               0 & 0 & -1 & 0\\
               -1 & -1 & -1 & -1\\
               0 & 1 & 1 & 1
			\end{matrix}
		\end{array}\right)\] and

\[\alpha(TST^{-1}) = \left(\begin{array}{c|c} 
			\begin{matrix}
			1 & 0 & 1 & 0 \\
            0 & 1 & 0 & 0\\
            0 & 0 & -1 & 0\\
            0 & 0 & 0 & 1 \\
			\end{matrix}  & 
			\begin{matrix} -1 & 0 & -1 & 0\\
                    0 & -1 & -1 & -1\\
                    1 & -1 & 1 & 2\\
                    0 & 0 & 1 & 0\\
                    \end{matrix} \\[0.5ex] & \\[-5.5ex] & \\[0.5ex] \hline & \\[-2ex]
			 \mbox{\Large $0$}  & 
			 \begin{matrix}
			  0 & 0 & -1 & -1\\
              1 & 1 & 1 &1\\
              0 & 0 & 1 & 0\\
              -1 & 0 & -1 & 0
			\end{matrix}
		\end{array}\right).\]
\end{thm}

\begin{proof}
We compute the action of $\Aff(M)$ on the homology of $M$. Recall that our generators are $S^{-1}T$ and $TST^{-1}$. Thus, it suffices to find $\alpha(S^{-1}T)$ and $\alpha(TST^{-1})$. We only write the computation for $\alpha(TST^{-1})$. The computation for $\alpha(S^{-1}T)$ is similar.

The following shows the action of $TST^{-1}$ on the homology.
\begin{figure}[htbp] 
	\centering
    \includegraphics[width=6in]{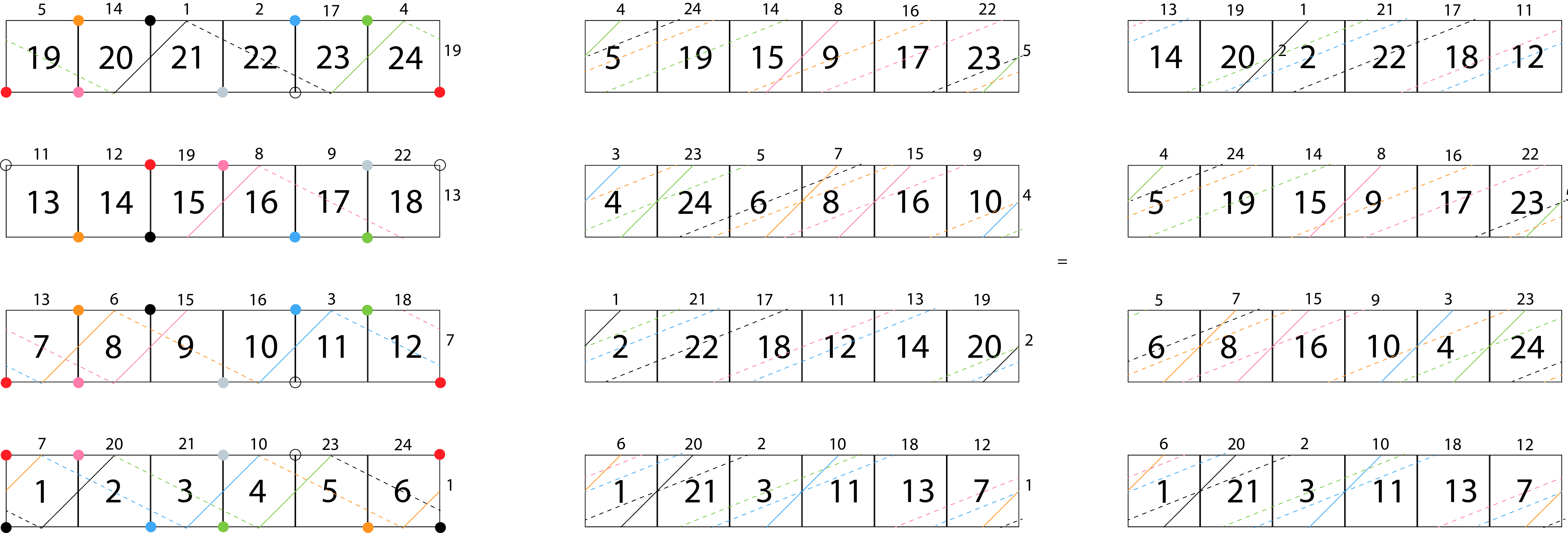}
    \caption{$M$ (left) and $TST^{-1}(M)$ (right)}
    \label{fig:TST^-1(MT4)}
\end{figure}

By utilizing the intersection form, we have action of $TST^{-1}$ acts on the basis of homology as:\begin{align*}
\tilde\alpha(TST^{-1})(\sigma_1) &= \sigma_1,\\  
\tilde\alpha(TST^{-1})(\sigma_3) &= \sigma_{3},\\
\tilde\alpha(TST^{-1})(\sigma_4) &= \sigma_1 -\sigma_4 + \sigma_8,\\   
\tilde\alpha(TST^{-1})(\sigma_{6}) &= \sigma_6,\\ 
\tilde\alpha(TST^{-1})(\sigma_{8}) &= \sigma_8,\\  
\tilde\alpha(TST^{-1})(\zeta_1) &= -2\sigma_1 + \frac{1}{2}\sigma_3 + \frac{1}{2}\sigma_4 - \frac{1}{2}\sigma_6 - \frac{3}{2}\sigma_8 + \frac{1}{2}\zeta_1 + \frac{1}{2}\zeta_3 - \frac{1}{2}\zeta_6 + \frac{1}{2}\zeta_8,\\
\tilde\alpha(TST^{-1})(\zeta_3) &=-\sigma_1 - \frac{1}{2}\sigma_3 - \frac{3}{2}\sigma_4 - \frac{1}{2}\sigma_6 + \frac{1}{2}\sigma_8 + \frac{1}{2}\zeta_1 + \frac{1}{2}\zeta_3 + \frac{1}{2}\zeta_6 - \frac{1}{2}\zeta_8,\\
\tilde\alpha(TST^{-1})(\zeta_4) &=-2\sigma_1 - \frac{1}{2}\sigma_3 + \frac{1}{2}\sigma_4 + \frac{1}{2}\sigma_6 - \frac{3}{2}\sigma_8 - \frac{1}{2}\zeta_1 + \frac{1}{2}\zeta_3 + \zeta_4 - \frac{1}{2}\zeta_6 + \frac{1}{2}\zeta_8 \\
\tilde\alpha(TST^{-1})(\zeta_6) &=-\sigma_1 - \frac{1}{2}\sigma_3 + \frac{3}{2}\sigma_4 - \frac{1}{2}\sigma_6 - \frac{5}{2}\sigma_8 - \frac{1}{2}\zeta_1 + \frac{1}{2}\zeta_3 + \frac{1}{2}\zeta_6 + \frac{1}{2}\zeta_8 \\
\tilde\alpha(TST^{-1})(\zeta_8) &= -\sigma_1 + \frac{1}{2}\sigma_3 - \frac{1}{2}\sigma_4 - \frac{1}{2}\sigma_6 - \frac{3}{2}\sigma_8 + \frac{1}{2}\zeta_1 - \frac{1}{2}\zeta_3 + \frac{1}{2}\zeta_6 + \frac{1}{2}\zeta_8.
\end{align*}

Using the above and recalling that the basis on the zero-holonomy subspace is given by $\Sigma_i \colonequals \sigma_i - \sigma_{8}$ and $Z_i  \colonequals \zeta_i - \zeta_8 $ for $i=1,3,4,6,8$, we have
\begin{align*}
    \tilde\alpha(TST^{-1})(\Sigma_1) &= \Sigma_1,\\
    \tilde\alpha(TST^{-1})(\Sigma_3) &= \Sigma_3,\\ 
    \tilde\alpha(TST^{-1})(\Sigma_4) &= \Sigma_1-\Sigma_4,\\ 
    \tilde\alpha(TST^{-1})(\Sigma_6) &= \Sigma_6,\\ 
    \tilde\alpha(TST^{-1})(Z_1) &= -\Sigma_1+\Sigma_4 +Z_3 -Z_6,\\
    \tilde\alpha(TST^{-1})(Z_3) &= -\Sigma_3-\Sigma_4 +Z_3,\\ 
    \tilde\alpha(TST^{-1})(Z_4) &= -\Sigma_1-\Sigma_3 +\Sigma_4+\Sigma_6 -Z_1 + Z_3 + Z_4 -Z_6,\\ 
    \tilde\alpha(TST^{-1})(Z_6) &= -\Sigma_3 +2\Sigma_4 -Z_1 + Z_3.
\end{align*}
This completes the proof.
\end{proof}

\subsection{Identification of the monodromy group}

\begin{thm} The Zariski closure of the monodromy group of the translation cover of the mutetrahedron is
$$\overline{\alpha(\Aff(M))}^{\Zar}\simeq \Sp(2,\R)^4.$$

\end{thm}
\begin{proof}
We begin by finding the isotypical components. Notice that the generators of $\Aut(M)$ are matrices of order 2 and commute with each other. This follows since $\Aut(M)\simeq E_8$ is abelian. It is easy to check that the eigenvalues of each are either $\pm1$. Since they commute, there exists a basis of $H_1(M;\R)$ of simultaneous eigenvectors. 

Let $$V(\lambda_1,\lambda_2,\lambda_3) = \{v\in H_1(M;\R):\rho(\pi_i)v=\lambda_i v, i=1,2,3\}$$
denote the simultaneous eigenspace corresponding to the ordered eigenvalues $\lambda_1,\lambda_2,\lambda_3$. Below we describe only the nontrivial spaces along with their bases,
\begin{align*}
V(-1,-1,1) &= \text{span}_\R\{-\sigma_1 ^4 +\sigma_4 ^4,-\zeta_3 ^4 + \zeta_6 ^4\}\\
V(-1,1,1) &= \text{span}_\R\{-\sigma_3 ^4 +\sigma_6 ^4,-\zeta_1 ^4 + \zeta_4 ^4\}\\
V(1,-1,-1) &= \text{span}_\R\{-\sigma_4 ^4 +\sigma_8 ^4,-\zeta_1 ^4 + \zeta_8 ^4\}\\
V(1,1,-1) &= \text{span}_\R\{-\sigma_1 ^4 + \sigma_3 ^4  + \sigma_6 ^4 - \sigma_8 ^4,\zeta_3 ^4 - \zeta_4 ^4 + \zeta_6 ^4 -\zeta_8 ^4 \}\\
V(1,1,1) &= \text{span}_\R\{-\sigma_1 ^4 - \sigma_8 ^4,-\zeta_4 ^4 - \zeta_8 ^4\}.
\end{align*}

These two-dimensional subspaces correspond to the isotypical components of $H_1(M;\R)$. Since the isoypical components are two-dimensional, the Zariski closure of the monodromy group $\tilde \alpha(\Aff(M))$ preserves each of these two-dimensional spaces and is no larger than $\Sp(2,\R)$. Since there are five such spaces, the monodromy group has dimension smaller than or equal to 15. Removing the contribution of the tautological subspace (which is always isomorphic to $\Sp(2,\R)$) implies that the Zariski closure of the monodromy group $\alpha(\Aff(M))$ is contained in the 12 dimensional Lie group $\Sp(2,\R)^4.$

On the other hand, we can find  12 elements in the Lie algebra of $\overline{\alpha(\Aff(M))}^{\Zar}$ which are linearly independent. Denote the generators of the Kontesevich--Zorich monodromy group of $M$ by $A=\alpha(S^{-1}T)$ and $B=\alpha(TST^{-1})$.

Since $B^2$ is parabolic, we consider $g=\log(B^2)$ of the Lie algebra $\overline{\alpha(\Aff(M))}^{\Zar}.$ 
Let $\phi_X$ denote the conjugation map $\phi_X(g)=XgX^{-1}$. The following 12 elements form a linearly independent set inside of the Lie algebra of $\overline{\alpha(\Aff(M))}^{\Zar}$:
\begin{align*}
    &g, \phi_{A}(g), \phi_{A^2}(g), \phi_{BA^2}(g),\phi_{ABA^2}(g),\phi_{A^2BA^2}(g),\phi_{(BA^2)^2}(g),\phi_{A(BA^2)^2}(g),\\
    &\phi_{A^2(BA^2)^2}(g),\phi_{(BA^2)^3 (BA)^2}(g),\phi_{(BA^2)^2A}(g),\phi_{A^2(BA)^2}(g)
\end{align*}
This shows that the Zariski closure of the monodromy group $\alpha(\Aff(M))$ has dimension greater than or equal to 12.
Thus, 
$$\overline{\alpha(\Aff(M))}^{\Zar}\simeq \Sp(2,\R)^4\simeq \SL(2,\R)^4.$$
\end{proof}
\subsection{Lyapunov exponents of the 2-cover of the mutetrahedron}

We compute the Lyapunov spectrum of the 2-cover of the mutetrahedron $M$.

\begin{prop}
Counting multiplicities, the positive Lyapunov spectrum of the translation cover of the mutetrahedron is
$$\{1\} \cup \{1/2\} \cup \{1/2\} \cup \{1/2\} \cup \{1/2\},$$
where the unions indicate the exponents corresponding to distinct, symplectically-orthogonal, irreducible pieces of the Hodge bundle.
\end{prop}

\begin{proof}
    Recall we have the following decomposition of the homology of $M$ into isotypical components,
    $$H_1(M;\mathbb R) = V(1,1,1)\oplus V(1,1,-1)\oplus V(1,-1,-1)\oplus V(-1,1,1)\oplus V(-1,-1,1).$$
    The subspace $V(1,1,1)$ is the tautological plane and carries a Lyapunov exponent of 1.
    
    Let $M_{i,j} = M/\langle \pi_i,\pi_j\rangle$ for $i,j\in\{1,2,3\}$. In case $i=j$ we denote $M_{i,j}$ simply by $M_i$. We note that $M_{1,2}$ and $M_{2,3}$ are genus 2 surfaces. Thus, by the Eskin--Kontsevich--Zorich formula \cite{MR3270590} and the \texttt{surface\_dynamics} package, we obtain that the sum of the Lyapunov exponents on $M_{1,2}$ and $M_{2,3}$ are $ 3/2$. That is, we obtain a Lyapunov exponent $1/2$ on $H_1(M_{1,2};\mathbb R)$ and $H_1(M_{2,3};\mathbb R)$.
    
    It remains to identify how these subspaces sit inside of $H_1(M;\mathbb R).$ 
    Let us consider $M_{1,2}$ first. The deck group of the covering $M \to M_{1,2}$ is precisely $\langle \pi_1, \pi_2\rangle$, and the desired subspace of $H_1(M;\mathbb R)$ is the one fixed by this deck group action. That is,
    \[
        H_1(M_{1,2};\mathbb R)=\{v\in H_1(M;\mathbb R): \pi_1v=v,\pi_2 v =v\}.
    \]
    
    However, since $\Aut(M)=\langle \pi_1,\pi_2,\pi_3\rangle$ is abelian, any simultaneous eigenvector of $\pi_1, \pi_2$ is an eigenvector of $\pi_3$. Since we computed all simultaneous eigenspaces previously, we know 
    \[
        H_1(M_{1,2};\mathbb R) = V(1,1,1)\oplus V(1,1,-1).
    \]
    Since the tautological plane $V(1,1,1)$ carries a Lyapunov exponent of 1, $V(1,1,-1)$ will carry a Lyapunov exponent of $1/2$.
    By considering the quotient $M_{2,3}$, we obtain that a Lyapunov exponent of 1/2 corresponds to $V(-1,1,1)$.

    It remains to show which Lyapunov exponents correspond to $V(-1,-1,1)$ and $V(-1,1,1)$. Consider $M_1$ and note that this is a genus 3 surface. As before, we have
    $$H_1(M_1;\mathbb R)=\{v\in H_1(M;\mathbb R): \pi_1v=v\}.$$
    Since $\Aut(M)$ is abelian, we have
    $$H_1(M_1;\mathbb R) = V(1,1,1)\oplus V(1,1,-1)\oplus V(1,-1,-1).$$
    From our calculation on $M_{1,2}$ we know that $V(1,1,-1)$ corresponds to a Lyapunov exponent $1/2$ and the tautological plane corresponds to 1. By applying the Eskin--Kontsevich--Zorich formula \cite{MR3270590} and the \texttt{surface\_dynamics} package, we conclude that $V(1,-1,-1)$ also corresponds to 1/2. 
    Lastly, by applying the Eskin--Kontsevich--Zorich formula \cite{MR3270590} and the \texttt{surface\_dynamics} package to $M$, the Lyapunov exponent 1/2 corresponds to $V(-1,1,1)$. 
\end{proof}



\nocite{*}
\begin{thebibliography}{10}

\bibitem{AA}
J.~S. Athreya and D.~Aulicino.
\newblock A trajectory from a vertex to itself on the dodecahedron.
\newblock {\em The American Mathematical Monthly}, 126:161 -- 162, 2019.

\bibitem{AAH}
J.~S. Athreya, D.~Aulicino, W.~P. Hooper, and with an appendix~by
  Anja~Randecker.
\newblock Platonic solids and high genus covers of lattice surfaces.
\newblock {\em Experimental Mathematics}, 0(0):1--31, 2020.

\bibitem{AL}
J.~S. Athreya and D.~Lee.
\newblock Translation covers of some triply periodic platonic surfaces.
\newblock {\em Conformal Geometry and Dynamics of the American Mathematical
  Society}, 25:34--50, 04 2021.

\bibitem{MR3717086}
A.~Avila, A.~Eskin, and M.~M\"{o}ller.
\newblock Symplectic and isometric {${\rm SL}(2,\Bbb R)$}-invariant subbundles
  of the {H}odge bundle.
\newblock {\em J. Reine Angew. Math.}, 732:1--20, 2017.

\bibitem{MR3959355}
A.~Avila, C.~Matheus, and J.-C. Yoccoz.
\newblock The {K}ontsevich-{Z}orich cocycle over {V}eech-{M}c{M}ullen family of
  symmetric translation surfaces.
\newblock {\em J. Mod. Dyn.}, 14:21--54, 2019.

\bibitem{MR2350698}
A.~Avila and M.~Viana.
\newblock Simplicity of {L}yapunov spectra: a sufficient criterion.
\newblock {\em Port. Math. (N.S.)}, 64(3):311--376, 2007.

\bibitem{MR2316268}
A.~Avila and M.~Viana.
\newblock Simplicity of {L}yapunov spectra: proof of the {Z}orich-{K}ontsevich
  conjecture.
\newblock {\em Acta Math.}, 198(1):1--56, 2007.

\bibitem{long}
E.~Bonnafoux, M.~Kany, P.~Kattler, C.~Matheus, R.~Niño, M.~Sedano-Mendoza,
  F.~Valdez, and G.~Weitze-Schmithüsen.
\newblock Arithmeticity of the {K}ontsevich--{Z}orich monodromies of certain
  families of square-tiled surfaces, 06 2022.

\bibitem{DFL}
V.~Delecroix, C.~Fougeron, and S.~Lelièvre.
\newblock Surface dynamics - sagemath package, version 0.4.1, 2019.

\bibitem{MR2820564}
A.~Eskin, M.~Kontsevich, and A.~Zorich.
\newblock Lyapunov spectrum of square-tiled cyclic covers.
\newblock {\em J. Mod. Dyn.}, 5(2):319--353, 2011.

\bibitem{MR3270590}
A.~Eskin, M.~Kontsevich, and A.~Zorich.
\newblock Sum of {L}yapunov exponents of the {H}odge bundle with respect to the
  {T}eichm\"{u}ller geodesic flow.
\newblock {\em Publ. Math. Inst. Hautes \'{E}tudes Sci.}, 120:207--333, 2014.

\bibitem{MR3424657}
A.~Eskin and C.~Matheus.
\newblock A coding-free simplicity criterion for the {L}yapunov exponents of
  {T}eichm\"{u}ller curves.
\newblock {\em Geom. Dedicata}, 179:45--67, 2015.

\bibitem{MR3619303}
S.~Filip.
\newblock Zero {L}yapunov exponents and monodromy of the {K}ontsevich-{Z}orich
  cocycle.
\newblock {\em Duke Math. J.}, 166(4):657--706, 2017.

\bibitem{MR3743240}
S.~Filip, G.~Forni, and C.~Matheus.
\newblock Quaternionic covers and monodromy of the {K}ontsevich-{Z}orich
  cocycle in orthogonal groups.
\newblock {\em J. Eur. Math. Soc. (JEMS)}, 20(1):165--198, 2018.

\bibitem{MR1153249}
W.~Fulton and J.~Harris.
\newblock {\em Representation theory}, volume 129 of {\em Graduate Texts in
  Mathematics}.
\newblock Springer-Verlag, New York, 1991.
\newblock A first course, Readings in Mathematics.

\bibitem{MR3959361}
R.~Guti\'{e}rrez-Romo.
\newblock A family of quaternionic monodromy groups of the
  {K}ontsevich-{Z}orich cocycle.
\newblock {\em J. Mod. Dyn.}, 14:227--242, 2019.

\bibitem{GJ}
E.~Gutkin and C.~Judge.
\newblock Affine mappings of translation surfaces: geometry and arithmetic.
\newblock {\em Duke Math. J.}, 103(2):191--213, 2000.

\bibitem{MR4120783}
P.~Hubert and C.~Matheus~Santos.
\newblock An origami of genus 3 with arithmetic {K}ontsevich-{Z}orich
  monodromy.
\newblock {\em Math. Proc. Cambridge Philos. Soc.}, 169(1):19--30, 2020.

\bibitem{L}
D.~Lee.
\newblock {\em Geometric realizations of cyclically branched coverings over
  punctured spheres}.
\newblock PhD thesis, Indiana University, 2018.

\bibitem{MR3402801}
C.~Matheus, M.~M\"{o}ller, and J.-C. Yoccoz.
\newblock A criterion for the simplicity of the {L}yapunov spectrum of
  square-tiled surfaces.
\newblock {\em Invent. Math.}, 202(1):333--425, 2015.

\bibitem{MYZ}
C.~Matheus, J.-C. Yoccoz, and D.~Zmiaikou.
\newblock Homology of origamis with symmetries.
\newblock {\em Ann. Inst. Fourier (Grenoble)}, 64(3):1131--1176, 2014.

\bibitem{MR0232867}
J.-P. Serre.
\newblock {\em Repr\'{e}sentations lin\'{e}aires des groupes finis}.
\newblock Hermann, Paris, 1967.

\end{thebibliography}

 \end{document}